\documentclass[10pt]{article}
\usepackage{amsmath,amsfonts,array,amscd,amsthm,amssymb,stackrel,verbatim,wrapfig,subfig,float}
\usepackage{enumerate,filecontents,microtype,authblk}
\usepackage[bookmarksopen=true]{hyperref}
\usepackage{graphicx,rotating}
\usepackage{url}
\usepackage{mathtools}
\usepackage{youngtab}
\usepackage{tikz}
\usepackage{tikz-cd}
\usetikzlibrary{arrows,matrix,automata,arrows,calc,positioning}
\usetikzlibrary{shapes.geometric}
\usetikzlibrary{decorations.markings} 

\tikzset{
  vertice/.style={circle,draw=black},
  decoration={markings,mark=at position 0.5 with {\arrow{>}}},
}

\newcommand{\C}{\mathbb{C}}
\newcommand{\N}{\mathbb{N}}
\newcommand{\Z}{\mathbb{Z}}
\newcommand{\R}{\mathbb{R}}
\newcommand{\A}{\mathbb{A}}
\newcommand{\PP}{\mathbb{P}}

\makeatletter

\newcommand{\Rmnum}[1]{\expandafter\@slowromancap\romannumeral #1@}
\makeatother
\def\multiset#1#2{\ensuremath{\left(\kern-.3em\left(\genfrac{}{}{0pt}{}{#1}{#2}\right)\kern-.3em\right)}}
\providecommand{\keywords}[1]{\textbf{\textit{Keywords:}} #1}
\theoremstyle{plain}
\newtheorem{theorem}{Theorem}[section]
\newtheorem{lemma}[theorem]{Lemma}
\newtheorem{proposition}[theorem]{Proposition}
\newtheorem{corollary}[theorem]{Corollary}

\newtheorem{example}[theorem]{Example}

\DeclareMathOperator{\Ext}{Ext}
\DeclareMathOperator{\Hom}{Hom}
\DeclareMathOperator{\trep}{\mathbf{trep}}
\DeclareMathOperator{\rep}{\mathbf{rep}}
\DeclareMathOperator{\diag}{diag}

\DeclareMathOperator{\Supp}{supp}
\DeclareMathOperator{\iss}{\mathbf{iss}}
\DeclareMathOperator{\GL}{GL}
\DeclareMathOperator{\SL}{SL}
\DeclareMathOperator{\PGL}{PGL}
\DeclareMathOperator{\PSL}{PSL}
\DeclareMathOperator{\M}{M}
\DeclareMathOperator{\B}{B}
\DeclareMathOperator{\Stab}{Stab}
\DeclareMathOperator{\Spec}{Spec}

\DeclareMathOperator{\Aut}{Aut}

\numberwithin{equation}{section}
\title{Representation theory of $\Z_2^{*n}$}
\author{Kevin De Laet}
\affil{Department of Mathematics-Computer Sciences,\\ University of Antwerp \\
Middelheimlaan 1, B-2020 Antwerpen (Belgium) \\ {\tt kevin.delaet2@uantwerpen.be}}
\date{}
\begin{document}

\maketitle
\begin{abstract}
We study the representations of the group $\Z_2^{*n}$, the free product of $\Z_2$ with itself $n$-times. We use the action of $B_n = S_2 \wr S_n $ as algebra automorphisms on the group algebra $\C(\Z_2^{*n})$ to find the components that contain simple representations and to study smoothness of their GIT-quotients. In particular, all the possible local quiver settings are studied for the component containing the standard $n$-dimensional representation of $S_{n+1}$.
\end{abstract}
\tableofcontents
\section{Introduction}
In \cite{bruyn2010trees}, a procedure was explained to study the representation theory of free products of semisimple algebras using quiver technology. The main focus of this paper was to study the representation theory of the aritmethic groups $\SL_2(\Z) \cong \Z_4 *_{\Z_2} \Z_6$ and $\PSL_2(\Z) \cong \Z_2 * \Z_3$, as was later done in \cite{bruyn2003one}. In particular, the fact that $\PSL_2(\Z)$ was a quotient of the third braid group $\B_3$ was used to study knot invertibility in \cite{bruyn2011dense}.
\par In this paper, we use this procedure to study the representation theory of 
$$
\Z_2^{*n} = \underbrace{\Z_2 * \Z_2 * \ldots * \Z_2}_{n \text{ times}}= \langle e_1, \ldots, e_n : e_i ^2=1\rangle.
$$
The motivation for studying the representation theory of this discrete family of groups are:
\begin{itemize}
\item for $n \geq 3$, $\Z_2^{*n}$ is a group of finite index in $\Z_2 * \Z_3$ by \cite{kulkarni1991arithmetic} (for $n=3$ it is a normal subgroup), and
\item the symmetric group $S_{n+1}$ is a quotient of $\Z_2^{*n}$ by taking the generators
$$
\{(1,i): 2\leq i \leq n+1\} \text{ or } \{(i,i+1): 1\leq i \leq n\}.
$$
\end{itemize}
To make calculations more manageable, one has $S_n \subset \Aut(\Z_2^{*n})$ or even $B_n \subset \Aut(\C(\Z_2^{*n}))$, $B_n$ being the hyperoctahedral group of order $n!2^n$.
\par In this paper, we will classify the components of $\rep_m \Z_2^{*n}$ that contain simple representations and prove that $\iss_m \Z_2^{*n}=\rep_m \Z_2^{*n}/\GL_m(\C)$ is almost never smooth.
\par As an application, we will study the component of $\rep_n \Z_2^{*n}$ containing the standard representation of $S_{n+1}$. This will correspond to the $\alpha(n,n)$-dimensional representations of a quiver $Q_n$ with dimension vector $\alpha(n,n)=(n-1,1)_{i=1}^n$. The main theorem of this section will be theorem \ref{theorem:mainintro}.
\begin{theorem}
The possible local quiver settings in $\dim \iss_{\alpha(n,n)}\Z_2^{*n}$ are determined by the following data:
\begin{itemize}
\item a partition $\mathbf{A}=\{A_i\}_{i=1}^l$ of $N$, and
\item a positive $l$-tuple $\mathbf{k}=(k_i)_{i=1}^l \in \N^l$ satisfying $1\leq k_i \leq |A_i|$ for all $1\leq i \leq l$.
\end{itemize}
A local quiver setting $(\mathbf{A},\mathbf{k})$ degenerates to $(\mathbf{A'},\mathbf{k'})$ if and only if
\begin{itemize}
\item $\mathbf{A'}$ is a refinement of $\mathbf{A}$, and
\item for all $1\leq i\leq l$, if $A_i = \bigsqcup_{j=1}^{l_i} A'_{i_j}$, then $k_i \geq \sum_{j=1}^{l_i} k'_{i_j}$. 
\end{itemize}
\label{theorem:mainintro}
\end{theorem}
In fact, we will describe all possible local quiver settings for the dimension vectors $\alpha(n,m)=(m-1,1)_{i=1}^n$, $1\leq m\leq n$.
\subsection{Notation} Some general notations will be used in the text:
\begin{itemize}
\item For a quiver $Q$ on $m$ vertices and a dimension vector $\alpha = (a_1,\ldots,a_m)$, we will denote with $\Supp(\alpha)$ the support of $\alpha$, that is, the full subquiver on those vertices $i$ such that $a_i \neq 0$.
\item We will see that $\rep_m \Z_2^{*n}$ is the union of ($\GL_m(\C)$-orbits of) open subsets of components in $\rep_{nm} Q_n = \sqcup_{|\alpha|=nm} \rep_\alpha Q_n$ for a certain quiver $Q_n$. The components of $\rep_{nm} Q_n$ that contain some representations of $\Z_2^{*n}$ are determined by dimension vectors $\alpha = (a_1^+,a_1^-;a_2^+,a_2^-;\ldots;a_n^+,a_n^-) = (a_i^+,a_i^-)_{i=1}^n$ such that $a_i^++a_i^- = m$ is constant for all $1\leq i \leq n$. We will then set 
$$
\rep_\alpha \Z_2^{*n} = \rep_m \Z_2^{*n} \cap \rep_\alpha Q_n. 
$$
It will be clear that we have $\rep_m \Z_2^{*n} = \GL_m(\C) \cdot (\sqcup_{|\alpha|=nm} \rep_\alpha \Z_2^{*n})$.
\item For a group $G$, we set $\iss_n G = \rep_n G/\GL_n(\C)$. Similarly, for a quiver $Q$ on $m$ vertices and a dimension vector $\alpha=(a_1,\ldots,a_m)$, we will write $\iss_\alpha Q = \rep_\alpha Q/\GL_\alpha(\C)$, with
$$
\GL_\alpha(\C) = \GL_{a_1}(\C) \times \GL_{a_2}(\C) \times \ldots \times \GL_{a_m}(\C)
$$
working on $\rep_\alpha Q$ by basechange in the vertices.
\item We will also write $\iss_\alpha \Z_2^{*n} = \rep_\alpha \Z_2^{*n}/\GL_\alpha(\C)$.
\item Let $A,B$ be 2 sets. Then we denote the symmetric difference of $A$ and $B$ by $A \Delta B$, which is equal to $(A \cup B) \setminus (A \cap B)$.
\item We will have to work with symmetric quivers. Therefore, we make the following conventions regarding arrows between two vertices:
\begin{itemize}
\item \begin{tikzpicture}[
    implies/.style={double,double equal sign distance,implies-implies},
    dot/.style={shape=circle,fill=black,minimum size=2pt,
                inner sep=0pt,outer sep=2pt}]
\node[vertice,circle] (a) at ( -1, 0) {};
\node[vertice,circle] (b) at (  1, 0) {};
\tikzset{every node/.style={fill=white},rectangle}
\draw[->,font=\scriptsize]
    (a) edge[<->] (b)
    ;
\end{tikzpicture}: one arrow back and forth,
\item \begin{tikzpicture}[
    implies/.style={double,double equal sign distance,implies-implies},
    dot/.style={shape=circle,fill=black,minimum size=2pt,
                inner sep=0pt,outer sep=2pt}]
\node[vertice,circle] (a) at ( -1, 0) {};
\node[vertice,circle] (b) at (  1, 0) {};
\tikzset{every node/.style={fill=white},rectangle}
\draw[->,font=\scriptsize]
    (a) edge[implies](b)
    ;
\end{tikzpicture}: two arrows back and forth, and
\item \begin{tikzpicture}[
    implies/.style={double,double equal sign distance,implies-implies},
    dot/.style={shape=circle,fill=black,minimum size=2pt,
                inner sep=0pt,outer sep=2pt}]
\node[vertice,circle] (a) at ( -1, 0) {};
\node[vertice,circle] (b) at (  1, 0) {};
\tikzset{every node/.style={fill=white},rectangle}
\draw[->,font=\scriptsize]
    (a) edge[implies] node[vertice,rectangle,anchor=center]{$k$}(b)
    ;
\end{tikzpicture}: $k$ arrows back and forth. The value of $k$ is omitted if $k=1,2$.
\end{itemize}
\end{itemize}

\section{The tame case $\Z_2 * \Z_2$}
\label{sec:tame}
The quiver $Q_2$ used for $\Z_2 * \Z_2$ is the following one
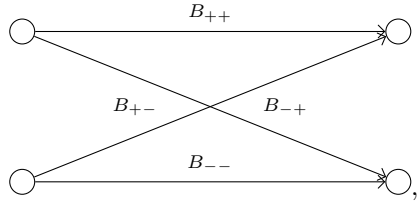
\begin{center}
\begin{tikzpicture}[scale=1]
   \node[vertice,circle] (a) at ( -5, 0) {};
   \node[vertice,circle] (b) at ( -5, -2) {};
   \node[vertice,circle] (c) at (  0, 0) {};
   \node[vertice,circle] (d) at (  0, -2) {};
   \path[->,font=\scriptsize,>=angle 90]
    (a) edge node[auto]{$B_{++}$} (c)
    (a) edge node[left of=a]{$B_{+-}$} (d)
    (b) edge node[right of=b]{$B_{-+}$} (c)
    (b) edge node[auto]{$B_{--}$} (d)
    ;
\end{tikzpicture},
\end{center}
which is a tame quiver. This fact is also visible in the group algebra $\C(\Z_2*\Z_2) = \C \langle x,y \rangle/(x^2-1,y^2-1)$. This group algebra is a Clifford algebra over its center $\C[T]$ with $T = \frac{1}{2}(xy+yx)$ and associated quadratic form over $\C[T]$ determined by the symmetric matrix
$$
\begin{bmatrix}
1 & T \\ T & 1
\end{bmatrix}.
$$
Accordingly, $\C(\Z_2*\Z_2)$ is a finite module over its center and a Cayley-Hamilton algebra of degree 2. Hence, the only simple representations are of dimension 1 or 2. The representation variety $\rep_2 \Z_2 * \Z_2$ is a disjoint union of 9 affine varieties,
$$
\rep_2 \Z_2 * \Z_2 = \bigsqcup_{(i,j) \in \Z_2^2} A_{i,j} \sqcup \bigsqcup_{(i,j) \in \Z_2^2} B_{i,j}\sqcup C, 
$$
with $\dim A_{i,j} = 0, \dim B_{i,j} = 2$ and $\dim C = 4$. We have $C = \trep_2 \C(\Z_2 * \Z_2)$, that is, the variety parametrizing trace preserving representations of $\C(\Z_2*\Z_2)$. Similarly, we find
$$
\iss_2 \Z_2 * \Z_2 = \{8 \text{ pts}\} \sqcup \A^1.
$$
We have an action of the dihedral group $D_4$ on $\C(\Z_2 * \Z_2)$ as algebra automorphisms, by taking the automorphism $(x,y) \mapsto (y,x)$ and $(x,y) \mapsto (-x,y)$. Consequently, $D_4$ acts on the 4 points $A_{i,j}$ as the 4 points of a square. The four components of dimension $2$ together form a $D_4$-variety, with each component a $\Z_2$-variety. The most interesting variety is $C$, which is a $D_4$-orbit. The one quiver $Q'_2$ is
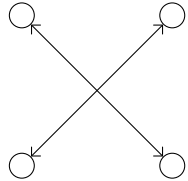
\begin{center}
\begin{tikzpicture}[scale=1]
   \node[vertice,circle] (a) at ( 1, 1) {};
   \node[vertice,circle] (b) at ( 1, -1) {};
   \node[vertice,circle] (c) at (  -1, 1) {};
   \node[vertice,circle] (d) at (  -1, -1) {};
   \path[->,font=\scriptsize,>=angle 90]
    (a) edge[<->] (d)
    (b) edge[<->] (c)
    ;
\end{tikzpicture}
\end{center}
This quiver is not connected (and is in fact the only disconnected one quiver for the free product of cyclic groups $\Z_p * \Z_q$). The four nodes correspond to the four one-dimensional representations $\psi_{\pm,\pm}$ defined by $\psi_{\pm,\pm}(x) = \pm 1, \psi_{\pm,\pm}(y) = \pm 1$. The arrows go from $\psi_{\pm \pm}$ to $\psi_{\mp,\mp}$.
\par The one quiver in this case is not necessary to parametrize open subsets of simple representations, as $A$ is a finite module over its center. On the downside, all simple representations of $\Z_2 * \Z_2$ are of dimension $\leq 2$.
\section{The set-up}
The quiver $Q_n$ necessary to describe the representations of $\Z_2^{*n}$ is the following: let $G_i$ be the subgroup of $\Z_2^{*n} = \langle e_1,\ldots,e_n : e_i^2=1 \rangle$ generated by $e_i$. Then $Q_n$ has $2n$ vertices, subdivided in pairs (the pair $T_i$ corresponds to the two one-dimensional representations of $G_i$). Pick $T_1$, then the arrows of $Q_n$ go from each vertex of $T_1$ to every other vertex of $Q_n$ belonging to some $T_i$, excluding $i=1$. We label the arrows accordingly $B^i_{kl}$, so $(k,l) \in \{+,-\}^2$ and $2 \leq i \leq n$, with the arrow $B^i_{kl}$ going from the $l$th vertex of $T_1$ to the $k$th vertex of $T_i$. The quiver $Q_n$ becomes figure \ref{fig2}.
\begin{figure}
\begin{center}
\begin{tikzpicture}[scale=2/5]
   \node[vertice,circle] (a) at ( -8, -6) {$a_1^+$};
   \node[vertice,circle] (b) at ( -8, -11) {$a_1^-$};
   \node[vertice,circle] (c) at (  0, -1) {$a_2^+$};
   \node[vertice,circle] (d) at (  0, -4) {$a_2^-$};
   \node[vertice,circle] (e) at (  0, -7) {$a_3^+$};
   \node[vertice,circle] (f) at (  0, -10) {$a_3^-$};
   \node[vertice,circle] (g) at (  0, -15) {$a_n^+$};
   \node[vertice,circle] (h) at (  0, -18) {$a_n^-$};   
   \draw[->,font=\scriptsize]
    (a) edge (c)
    (a) edge (d)
    (a) edge (e)
    (a) edge (f)
    (a) edge (g)
    (a) edge (h)
    (b) edge (c)
    (b) edge (d)
    (b) edge (e)
    (b) edge (f) 
    (b) edge (g)
    (b) edge (h)   
    ;
   \draw[dotted,font=\scriptsize]
    (f) edge (g);
\end{tikzpicture}
\end{center}
\caption{The associated quiver $Q_n$ to $\Z_2^{*n}$}
\label{fig2}
\end{figure}
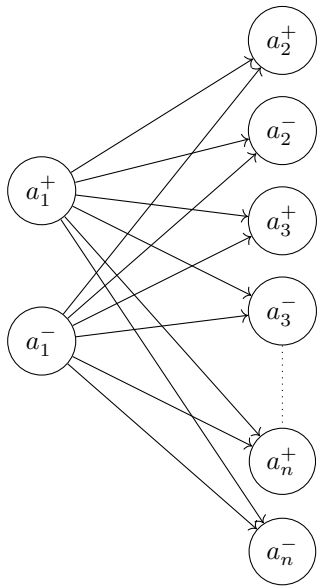
\par To an $\alpha$-dimensional representation of $Q_n$ with $\alpha=(a_i^+,a_i^-)_{i=1}^n$ a dimension vector fulfilling $a_i^+ +a_i^-=m$ for all $1 \leq i \leq n$ corresponds the $m$-dimensional representation $\phi$ of $\Z_2^{*n}$ defined by
$$
\phi(e_1)= \begin{bmatrix}
1_{a_1^+} & 0 \\0 & -1_{a_1^-}
\end{bmatrix},
\phi(e_i)= \begin{bmatrix}
B_{++}^i & B_{+-}^i \\ B_{-+}^i & B_{--}^i
\end{bmatrix}^{-1} \begin{bmatrix}
1_{a_i^+} & 0 \\0 & -1_{a_i^-}
\end{bmatrix}
\begin{bmatrix}
B_{++}^i & B_{+-}^i \\ B_{-+}^i & B_{--}^i
\end{bmatrix}
$$
for $2 \leq i \leq n$.
\par The hyperoctahedral group $B_n = S_2 \wr S_n = \Z_2^n \rtimes S_n$ acts on the group algebra $\C(\Z_2^{*n})$ as automorphisms, the $S_n$-action permuting the generators $e_i$ and the $S_2$-action coming from the algebra automorphism $\varphi(e_1) = -e_1, \varphi(e_i) = e_i, 2 \leq i \leq n$. It is clear that the same group also acts on the dimension vectors $\alpha = (a_i^+,a_i^-)_{i=1}^n$, with $S_n$ permuting the pairs $(a_i^+,a_i^-)$ and $\varphi$ switching $a_1^+$ with $a_1^-$ and leaving the rest fixed.
\par We can already analyze the number of components of $\rep_m \Z_2^{*n}$.
\begin{theorem}
The variety $\rep_m \Z_2^{*n}$ is the disjoint union of $(m+1)^n$ components, labeled by $\{0,1,\ldots,m\}^n$. The group $B_n$ acts on these components, with the number of $B_n$-orbits equal to
\begin{itemize}
\item $\multiset{\frac{m}{2}+1}{n} = \binom{\frac{m}{2}+n}{n}$ if $m$ is even, and
\item $\multiset{\frac{m+1}{2}}{n} = \binom{\frac{m-1}{2}+n}{n}$ if $m$ is odd.
\end{itemize}
\label{th:numberofcomponent}
\end{theorem}
\begin{proof}
The possible dimension vectors for $Q_n$ are determined by choosing for each $1\leq i \leq n$ an element $0\leq a_i^+ \leq m$, because then $a_i^- = m-a_i^+$. This gives $(m+1)^n$ components, labeled by $\{0,1,\ldots,m\}^n$.
\par By the $B_n$-action, each dimension vector lies in a unique $B_n$-orbit of a dimension vector $(a_i^+,a_i^-)_{i=1}^n$ such that
$$
m\geq a_1^+ \geq a_2^+ \geq \dots a_n^+ \geq \frac{m}{2} \text{ or } \frac{m+1}{2}.
$$
This leads to the claimed multiset coefficients.
\end{proof}
\subsection{The one quiver $Q'_n$}
\par The generators of the Abelian semigroup
$$
S(n)=\{ (a_i^+,a_i^-)_{i=1}^n \in \N^{2n}: a_1^+ + a_1^-=a_2^+ + a_2^-=\ldots = a_n^+ + a_n^-\},
$$
are the dimension vectors in the $B_n$-orbit of $(1,0)_{i=1}^n$ and
correspond to the simple representations of $\Z_2^n = \Z_2^{*n}/[\Z_2^{*n},\Z_2^{*n}]$, so there are $2^n$ generators. Set $S(n)_m = \{\alpha \in S(n):a_1^+ + a_1^-=m\}$. Let $N= \{1,\ldots,n\}$. Then we can label all characters with the elements of $2^{N}$. For $A \in 2^N$, let $\psi_A$ be the character of $\Z_2^n$ corresponding to $\psi_A(e_i)=-1$ if and only if $i \in A$. In terms of generators of the semigroup, $\psi_A$ corresponds to the dimension vector $\alpha_A = (a_i^+(A),a_i^-(A))_{i=1}^n$ with
\begin{gather*}
i \not\in A \Leftrightarrow (a_i^+(A),a_i^-(A))=(1,0),\\
i \in A \Leftrightarrow (a_i^+(A),a_i^-(A))=(0,1).
\end{gather*}
Regarding $S(n)$ and the generators $\alpha_A$, we find
\begin{proposition}
Let $F_n$ be the free multiplicative commutative semigroup generated by $e_A, A \in 2^N$, graded by $\deg(e_A)=1$. 
Then the semigroup $S(n)$ is a quotient of $F_n$ by the relations 
$$
e_A e_B = e_{A \cup B} e_{A \cap B}
$$
for all $A,B \in 2^N$.
\end{proposition}
\begin{proof}
There is a relation $\sum_{A \in 2^N} x_A \alpha_A = \sum_{A \in 2^N} y_A \alpha_A$ with $x_A,y_A \in \N$ if and only if 
$$
\sum_{A \in 2^N} \psi_A^{x_A}, \sum_{A \in 2^N} \psi_A^{y_A} \in \rep_\alpha \Z_2^{*n}
$$
for some $\alpha \in S(n)$. From theorem \ref{th:numberofcomponent} it follows that $|S(n)_m| = (m+1)^n$. It is easy to see that $\alpha_A + \alpha_B =  \alpha_{A\cup B} + \alpha_{A \cap B}$ is indeed true for each $A,B \in 2^N$.
\par Let $\mathcal{A}(n)=\C[F_n/\langle e_A e_B =  e_{A \cup B} e_{A \cap B}\rangle]$. Then $\mathcal{A}(n)$ is the homogeneous coordinate ring of the Segre embedding $\begin{tikzcd}
(\PP^1)^n \arrow[r] & \PP^{2^n-1}
\end{tikzcd}$ by \cite[Corollary 1.8]{ha2002box}. Consequently,
$$\dim (\mathcal{A}(n))_m = \C[X_1,Y_1;X_2,Y_2; \ldots ; X_n,Y_n]_{(m,m,\ldots,m)}=(m+1)^n.
$$
This implies that $F_n/\langle e_A e_B =  e_{A \cup B} e_{A \cap B}\rangle = S(n)$, as claimed.
\end{proof}
\par We will also denote the corresponding vertex in the one quiver $Q'_n$ by $\psi_A$. The group $B_n$ acts on these $2^n$ generators in the following way: $S_n$ acts on the elements of $2^N$ by permuting the elements of each set, and $\varphi\cdot A = A \Delta \{e_1\}$.
\par  In order to calculate the one quiver, it is enough to calculate all the outgoing arrows from one vertex and use the $B_n$-symmetry to find the rest. Let us take the trivial character $\psi_\emptyset$. It is clear that the dimension vector of $Q_n$ corresponding to $\psi_\emptyset$ is $(1,0)_{i=1}^n$.
\begin{theorem}
The number of arrows going from $\psi_\emptyset$ to $\psi_A$, $A \in 2^N$ is equal to $|A|-1$.
\end{theorem}
\begin{proof}The matrix determining the Euler form $\chi_{Q_n}(-,-)$ of $Q_n$ is equal to
$$
\mathcal{M}_n=\begin{bmatrix}
1      & 0      & -1     & -1     & \ldots & -1 \\
0      & 1      & -1     & -1     & \ldots & -1 \\
0      & 0      &  1     &  0     & \ldots & 0 \\
0      & 0      &  0     &  1     &        & 0 \\
\vdots & \vdots &        &        & \ddots & \vdots\\
0      & 0      & \ldots &        &        & 1
\end{bmatrix}.
$$
Calculating $(1,0;1,0;\ldots;1,0)M$, we find $(1,0;0,-1;0,-1;\ldots;0,-1)$.
\par Using the formula $\dim \Hom_{\Z^{*n}}(V,W) - \dim \Ext^1_{\Z^{*n}}(V,W) = \chi_{Q_n}(V,W)$ for representations $V$ and $W$ of $Q_n$, we find for any $A \in 2^N$
\begin{align*}
\dim \Ext^1_{\Z^{*n}}(\psi_\emptyset,\psi_\emptyset)&= 0,\\
\dim \Ext^1_{\Z^{*n}}(\psi_\emptyset,\psi_A)&= [A] - 1.
\end{align*}
\end{proof}
\par Using the $B_n$-symmetry, we find that for two set $A,B \in 2^N$ with $A \neq B$
\begin{align*}
\dim \Ext^1_{\Z^{*n}}(\psi_A,\psi_A)&= 0,\\
\dim \Ext^1_{\Z^{*n}}(\psi_A,\psi_B)&=|\{e_i | \psi_A(e_i) \neq \psi_B(e_i)\}| -1 = |A \Delta B|-1.
\end{align*} 

\begin{corollary}
Let $2^N$ be the set of vertices of the unit hypercube in $\R^n$. Then the one quiver $Q'_n$ for $\Z_2^{*n}$ has as vertices the elements of $2^N$, no loops and $k-1$ arrows between two sets $A\neq B \in 2^N$ if and only if $k = d_E(s,t)^2= |A \Delta B|$, with $d_E$ the Euclidean distance in $\R^n$.
\label{cor:onequiver}
\end{corollary}
It is easy to find an inductive procedure to determine the Euler form of $Q'_n$.
\begin{proposition}
If $M_{n-1}$ is the matrix determining the Euler form of the one quiver $Q'_{n-1}$ for $\Z_2^{*(n-1)}$ by the lexicographical order (with $\emptyset < N$), then the Euler form of the one quiver $Q'_n$ for $\Z_2^{*n}$ is determined by the matrix
$$
M_n=
\begin{bmatrix}
M_{n-1} & M_{n-1} - P_{n-1} \\ M-P_{n-1} & M_{n-1}
\end{bmatrix},
$$
with $P_{n-1}$ is the $(n-1) \times (n-1)$-matrix with every entry equal to $1$.
\end{proposition}
\begin{proof}
Any $A \in 2^N$ is either a subset of $\{1,\ldots,n-1\}$ or $n \in A$. Let $A$ and $B$ be 2 subsets of $N$. There are two cases to consider:
\begin{itemize}
\item $A,B \subset \{1,\ldots,n-1\}$ or $n \in A\cap B$: then $A \Delta B \subset \{1,\ldots,n-1\}$, so
$$\dim \Ext^1_{\Z_2^{*n}}(\psi_A,\psi_B) = \dim\Ext^1_{\Z_2^{*(n-1)}}(\psi_A|_{\Z_2^{*(n-1)}},\psi_B|_{\Z_2^{*(n-1)}}).$$
\item $A \subset \{1,\ldots,n-1\}$ and $B \not\subset \{1,\ldots,n-1\}$: in this case we have $A \Delta B = (A \Delta (B\cap\{1,\ldots,n-1\})) \cup \{n\}$, so we have
$$
\dim \Ext^1_{\Z_2^{*n}}(\psi_A,\psi_B) = \dim\Ext^1_{\Z_2^{*(n-1)}}(\psi_A|_{\Z_2^{*(n-1)}},\psi_B|_{\Z_2^{*(n-1)}})+1.
$$
\end{itemize}
The claim follows.
\end{proof}
\begin{corollary}
If $M_n$ is the Euler matrix of $Q'_n$, then
$$
(M_n)_{A,B} = 1-|A \Delta B|.
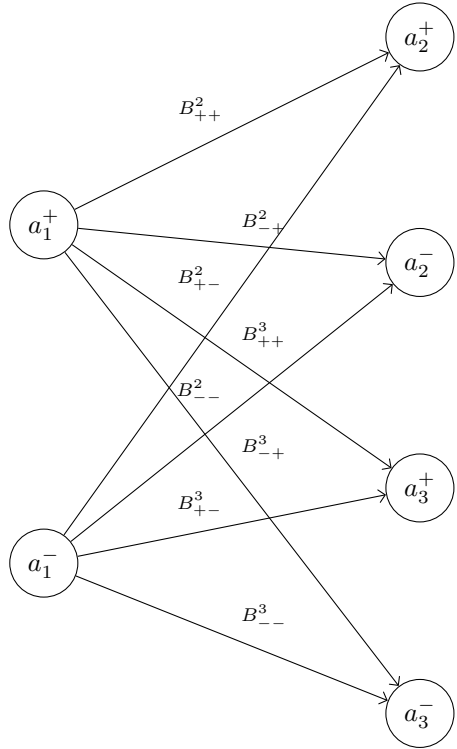
$$
\end{corollary}
\begin{center}
\begin{tikzpicture}[scale=1]
   \node[vertice,circle] (a) at ( -5, -3.5) {$a_1^+$};
   \node[vertice,circle] (b) at ( -5, -8)   {$a_1^-$};
   \node[vertice,circle] (c) at (  0, -1)   {$a_2^+$};
   \node[vertice,circle] (d) at (  0, -4)   {$a_2^-$};
   \node[vertice,circle] (e) at (  0, -7)   {$a_3^+$};
   \node[vertice,circle] (f) at (  0, -10)  {$a_3^-$};
   \path[->,font=\scriptsize,>=angle 90]
    (a) edge node[auto]{$B^2_{++}$} (c)
    (a) edge node[auto]{$B^2_{-+}$} (d)
    (a) edge node[auto]{$B^3_{++}$} (e)
    (a) edge node[auto]{$B^3_{-+}$} (f)
    (b) edge node[auto]{$B^2_{+-}$} (c)
    (b) edge node[auto]{$B^2_{--}$} (d)
    (b) edge node[auto]{$B^3_{+-}$} (e)
    (b) edge node[auto]{$B^3_{--}$} (f)    ;
\end{tikzpicture}\\
\end{center}
\begin{example}
A representation of the group $\Z_2^{*3}=\Z_2 * \Z_2 * \Z_2$ corresponds to a representation of the quiver $Q_3$ for a dimension vector $\alpha = (a_i^+,a_i^-)_{i=1}^3$  satisfying $a_1^++a_1^- = a_2^+ + a_2^- = a_3^+ + a_3^-$, on the condition that the matrices 
$$
B(2)=\begin{bmatrix}
B^2_{++} & B^2_{+-} \\ B^2_{-+} & B^2_{--}
\end{bmatrix} \text{ and }B(3)=\begin{bmatrix}
B^3_{++} & B^3_{+-} \\ B^3_{-+} & B^3_{--}
\end{bmatrix}$$
are invertible. In particular, the useful representations of this quiver form an open subset of $\rep_\alpha Q_3$. A representation of $Q_3$ fulfilling the invertibility condition then leads to the following three matrices that determine a representation of $\Z_2^{*3}$
\begin{align*}
\left(\begin{bmatrix}
1_{a_1^+} & 0 \\ 0 & -1_{a_1^-}
\end{bmatrix},
B(2)^{-1} \begin{bmatrix}
1_{a_2^+} & 0 \\ 0 & -1_{a_2^-}
\end{bmatrix}B(2), B(3)^{-1}\begin{bmatrix}
1_{a_3^+} & 0 \\ 0 & -1_{a_3^-}
\end{bmatrix}B(3)\right).
\end{align*}
\par In this case, the one quiver $Q'_3$ has as vertices the vertices of the unit cube in $\R^3$ and as arrows one arrow in each direction for each side diagonal and two arrows in each direction for each space diagonal. The corresponding one quiver is
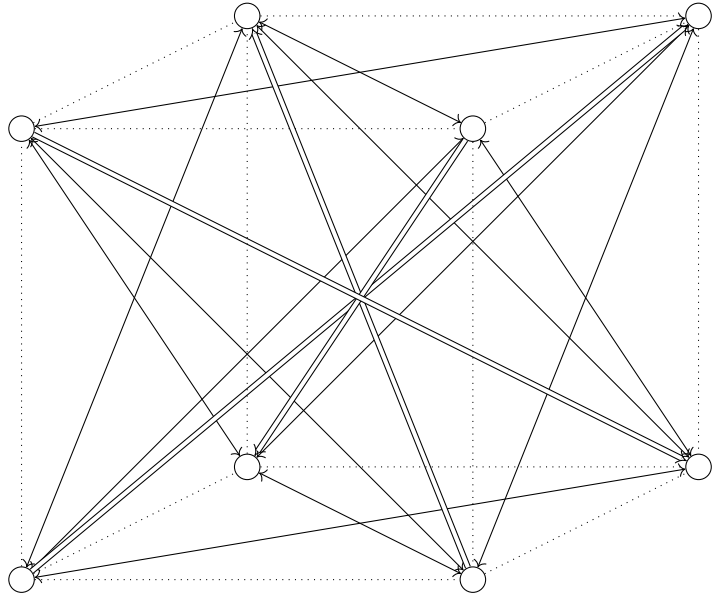
\begin{center}
\begin{tikzpicture}[
    implies/.style={double,double equal sign distance,-implies},
    dot/.style={shape=circle,fill=black,minimum size=2pt,
                inner sep=0pt,outer sep=2pt},
scale=3]
   \node[vertice,circle] (a) at ( 0, 0) {};
   \node[vertice,circle] (b) at ( 2, 0) {};
   \node[vertice,circle] (c) at ( 2, 2) {};   
   \node[vertice,circle] (d) at ( 0, 2) {};  
   \node[vertice,circle] (e) at ( 1, 1/2) {};
   \node[vertice,circle] (f) at ( 3, 1/2) {};
   \node[vertice,circle] (g) at ( 3, 5/2) {};   
   \node[vertice,circle] (h) at ( 1, 5/2) {}; 
   \draw[<->,font=\scriptsize]
   (a) edge (c)
   (b) edge (d)
   (a) edge (f)
   (a) edge (h)
   (d) edge (e)
   (b) edge (e)
   (b) edge (g)
   (c) edge (f)
   (f) edge (h)
   (e) edge (g)
   (d) edge (g)
   (c) edge (h)
   (a) edge[implies] (g)
   (b) edge[implies] (h)
   (c) edge[implies] (e)
   (d) edge[implies] (f);
   \draw[dotted,-,font=\scriptsize]
   (a) edge (e)
   (a) edge (b)
   (a) edge (d)
   (d) edge (c)
   (d) edge (h)
   (c) edge (b)
   (c) edge (g)
   (b) edge (f)
   (e) edge (f)
   (e) edge (h)
   (h) edge (g)
   (g) edge (f); 
\end{tikzpicture}
\end{center}
The corresponding Euler matrix of $Q'_3$ is
$$
M_3=
\begin{bmatrix}
1 & 0 & 0 & -1 & 0 & -1 & -1 & -2 \\
0 & 1 & -1 & 0 & -1 & 0 &-2 & -1 \\
0 & -1 & 1 & 0 & -1 &-2 & 0 & -1 \\
-1 & 0 & 0 & 1 & -2 &-1 & -1 & 0 \\
0 & -1 & -1 & -2 & 1 & 0 & 0 & -1 \\
-1 & 0 &-2 & -1 & 0 & 1 & -1 & 0  \\
-1 &-2 & 0 & -1 & 0 & -1 & 1 & 0 \\
-2 &-1 & -1 & 0 & -1 & 0 & 0 & 1
\end{bmatrix}.
$$
\end{example}

\subsection{Components with simple representations}
We assume that $n\geq 3$ and that $m \geq 2$.
\par We have seen in theorem \ref{th:numberofcomponent} that the variety $\rep_m \Z_2^{*n}$ is a disjoint union of $(m+1)^n$ affine varieties $\sqcup_{|\alpha|=nm} \rep_\alpha \Z_2^{*n}$.
\par This subsection will classify the dimension vectors for which $\rep_\alpha \Z_2^{*n}$ contains simple representations. Fix a dimension vector $\alpha = (a_i^+,a_i^-)_{i=1}^n$ with $m = a_1^++a_1^- = a_2^++a_2^- = \ldots = a_n^++a_n^-$ for $Q_n$. Let $M$ be a semisimple representation belonging to $\rep_\alpha \Z_2^{*n}$, such that $M$ decomposes as a direct sum of one-dimensional representations
$$
M = \oplus_{A \in 2^N} \psi_A^{b_A}.
$$
We know that there is an \'etale morphism such that the image contains an open subset (see amongst others \cite[Theorem 3.14]{LBBook})
$$
\begin{tikzcd}
\GL_\alpha(\C) \times^{\Stab(M)} \rep_\beta Q'_n \arrow{r} & \rep_\alpha \Z_2^{*n}
\end{tikzcd}
$$
where $Q'_n$ is the one quiver we calculated in the previous section and $\beta=(b_A)_{A \in 2^N}$ is a dimension vector for $Q'_n$. It is clear that
$$
\Stab(M) = \prod_{A \in 2^N} \GL_{b_A}(\C).
$$
In order for $\rep_\alpha \Z_2^{*n}$ to contain simple representations, we need to find dimension vectors $\beta$ for $Q'_n$ such that $\rep_\beta Q'_n$ contains simple representations of $Q'_n$. Recall the following result:
\begin{theorem}\cite[Theorem 4.10]{bruyn2003one}
Let $Q$ be a quiver with associated Euler form $\chi_Q$. Let $\alpha = (d_1, \ldots,d_k)$ be a dimension vector for $Q$. Then $\rep_\alpha Q$ contains simple representations if and only if one of the following conditions is satisfied
\begin{itemize}
\item $\Supp(\alpha)$ is an extended Dynkin quiver of type $\tilde{A}_k$ on $k$ vertices with cyclic orientation and $\alpha= (1,\ldots,1)$,
\item $\Supp(\alpha)$ is not of type $\tilde{A}_k$. Then $\Supp(\alpha)$ is stronly connected, $\chi_Q(\alpha,\epsilon_i) \leq 0$ and $\chi_Q(\epsilon_i,\alpha) \leq 0$ for all $1\leq i \leq k$.
\end{itemize}
\label{th:simpledim}
\end{theorem}
Recall that a strongly connected quiver is a quiver $Q$ such that for each pair of vertices $v_i$, $v_j$ of $Q$, there exists an oriented cycle in $Q$ passing through $v_i$ and $v_j$.
\par As all full subquivers of $Q'_n$ are symmetric, the first case can only be if $\Supp(\beta)$ is the full subquiver on $\psi_A$ and $\psi_B$ with $A \Delta B = 2$ and $b_A = b_B = 1$. If this is not the case, then we need to check that $\Supp(\beta)$ is strongly connected and that $\chi_{Q'_n}(\beta,\psi_A) \leq 0$ for each $\psi_A\in \Supp(\beta)$, as $Q'_n$ is symmetric. As $Q'_n$ is symmetric, a full subquiver is strongly connected if and only if it is connected.
\begin{lemma}
Let $Q^1$ and $Q^2$ be two full subquivers of $Q'_n$ such that the number of vertices of $Q^1$ and the number of vertices of $Q^2$ are both separately larger than or equal to 2. If $Q^1 \sqcup Q^2 \subset Q'_n$, then the underlying graph is
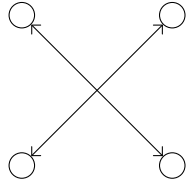
\begin{center}
\begin{tikzpicture}[scale=1]
   \node[vertice,circle] (a) at ( 1, 1) {};
   \node[vertice,circle] (b) at ( 1, -1) {};
   \node[vertice,circle] (c) at (  -1, 1) {};
   \node[vertice,circle] (d) at (  -1, -1) {};
   \path[->,font=\scriptsize,>=angle 90]
    (a) edge[<->] (d)
    (b) edge[<->] (c)
    ;
\end{tikzpicture}
\end{center}
\label{lemma:disconnected}
\end{lemma}
\begin{proof}
By the $B_n$-action, we may assume that one of the vertices of $Q^1$ is $\psi_\emptyset$ and that one of the vertices of $Q^2$ is $\psi_{\{1\}}$, as it has to correspond to a set with one element, which are all equivalent under the $S_n$-action. By assumption, any other vertex of $Q^2$ is not connected to $\psi_{\emptyset}$, which can only be if those vertices form a subset of $\{\psi_{\{k\}}:k \in N, k \neq 1\}$. Again by the $B_n$-action, we may assume that $\psi_{\{2\}}$ is one of the vertices of $Q^2$.
\par Now, any vertex other than $\psi_\emptyset$ in $Q^1$ corresponds to a set $\emptyset \neq A \in 2^N$ such that
$$
|A \Delta \{1\}| = |A \Delta \{2\}| = 1.
$$
There is only one set with this property, namely $A = \{1,2\}$. But then no other vertices are also possible in $Q^2$, for 
$$
|\{1,2\} \Delta \{k\}| = 3 \text{ if } k\neq 1,2.
$$
\end{proof}
Consequently, if $Q^1 \sqcup Q^2 \subset Q'_n$ is a full subquiver with the number of vertices of $Q^1$ strictly larger than 2, then $Q^2$ has only one vertex.
\begin{lemma}
For each vertex in $Q'_n$, there are exactly $n$ vertices not connected to it. The full subquiver on these $n$ vertices is the complete directed symmetric graph on $n$ points.
\end{lemma}
\begin{proof}
This is true for $\psi_\emptyset$, for the other vertices one can use the $B_n$-action.
\end{proof}
\begin{corollary}
There are $n+1$ types of disconnected full subquivers in $Q'_n$:
\begin{itemize}
\item $Q^1 \sqcup Q^2$ with the number of vertices of $Q^1$ and of $Q^2$ equal to 2, connected as in lemma \ref{lemma:disconnected},
\item $Q^1 \sqcup Q^2$, with the number of vertices of $Q^1$ equal to $k$ for some $1\leq k \leq n$ and $Q^2$ the quiver with one vertex and no loops. In this case, $Q^1$ is the complete directed graph on $k$ vertices.
\end{itemize}
\end{corollary}
\begin{proposition}
$\beta$ is a simple dimension vector of $Q'_n$ if and only if 
\begin{align}
|\beta| \leq \sum_{B \in 2^N} b_B |A \Delta B|
\label{conditionsimple}
\end{align}
for all $A \in 2^N$, unless we are in the quiver setting
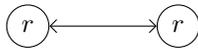
\begin{center}
\begin{tikzpicture}
\node[vertice,circle] (a) at ( -1, 0) {$r$};
\node[vertice,circle] (b) at ( 1, 0) {$r$};
\path[->,font=\scriptsize,>=angle 90]
    (a) edge[<->]  (b)
    ;
\end{tikzpicture}
\end{center}
which is a simple dimension vector if and only if $r=1$ or in the quiver setting
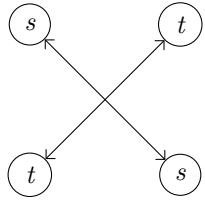
\begin{center}
\begin{tikzpicture}[scale=1]
   \node[vertice,circle] (a) at ( 1, 1) {t};
   \node[vertice,circle] (b) at ( 1, -1) {s};
   \node[vertice,circle] (c) at (  -1, 1) {s};
   \node[vertice,circle] (d) at (  -1, -1) {t};
   \path[->,font=\scriptsize,>=angle 90]
    (a) edge[<->] (d)
    (b) edge[<->] (c)
    ;
\end{tikzpicture}
\end{center}
with $st \neq 0$, which is never a simple dimension vector.
\label{prop:simplocal}
\end{proposition}
\begin{proof}
First, we will show that if $\Supp(\beta)$ is the quiver from lemma \ref{lemma:disconnected}, then $\beta$ satisfies condition \ref{conditionsimple} if and only if $\beta$ lies in the $B_n$-orbit of the dimension vector $(b_\emptyset,b_{\{1\}},b_{\{2\}},b_{\{1,2\}},0,0,\ldots,0)$ with $b_\emptyset=b_{\{1,2\}}$, $b_{\{1\}}=b_{\{2\}}$. We may assume by the $B_n$-action that $\beta = (b_\emptyset,b_{\{1\}},b_{\{2\}},b_{\{1,2\}},0,0,\ldots,0)$ fulfils \ref{conditionsimple}, with $b_\emptyset b_{\{1\}}b_{\{2\}}b_{\{1,2\}}\neq 0$. Set $A$ equal to respectively $\emptyset,\{1\},\{2\},\{1,2\}$, then this implies
\begin{align*}
b_\emptyset+b_{\{1\}}+b_{\{2\}}+b_{\{1,2\}}&\leq                   b_{\{1\}} + b_{\{2\}} + 2b_{\{1,2\}},\\
b_\emptyset+b_{\{1\}}+b_{\{2\}}+b_{\{1,2\}}&\leq   b_\emptyset + 2 b_{\{2\}} + b_{\{1,2\}},\\
b_\emptyset+b_{\{1\}}+b_{\{2\}}+b_{\{1,2\}}&\leq   b_\emptyset + 2 b_{\{1\}} + b_{\{1,2\}},\\
b_\emptyset+b_{\{1\}}+b_{\{2\}}+b_{\{1,2\}}&\leq 2 b_\emptyset +   b_{\{1\}} + b_{\{2\}}.
\end{align*}
From this it follows that $b_\emptyset = b_{\{1,2\}}$ and $b_{\{1\}} = b_{\{2\}}$. Condition \ref{conditionsimple} therefore becomes:
$$
2b_\emptyset+2b_{\{1\}} \leq b_{\emptyset}(|A|+|A\Delta\{1,2\}|)+ b_{\{1\}}(|A\Delta\{1\}|+|A\Delta\{2\}|).
$$
Therefore, for the other $A \in 2^N$, we find:
\begin{itemize}
\item If $|A|\geq 3$: then $|A|\geq 3$ and $|A\Delta\{1\}|\geq 2$, so condition \ref{conditionsimple} is satisfied.
\item If $|A| = 2$: if $A \cap \{1,2\} = \emptyset$ then \ref{conditionsimple} is again satisfied. If $|A \cap \{1,2\}| = 1$, then either $|A \Delta \{1\}|\geq 3$ or $|A \Delta \{2\}|\geq 3$ so condition \ref{conditionsimple} is again satisfied.
\item If $|A|=1$, then $|A\Delta\{1,2\}| \geq 3$ and $|A\Delta\{1\}|\geq 2$, so also in this case condition \ref{conditionsimple} is satisfied.
\end{itemize}

\par Now, if $\Supp(\beta) = Q^1 \sqcup Q^2$ with $Q^1$ consisting of just one vertex, then we may assume that $\beta = (b_\emptyset,b_{\{1\}},b_{\{2\}},b_{\{3\}},\ldots,b_{\{k\}},0,\ldots,0)$ for some $k \leq n$. For $A = \emptyset$, condition \ref{conditionsimple} would say
$$
b_\emptyset+b_{\{1\}}+b_{\{2\}}+b_{\{3\}}+\ldots,b_{\{k\}} \leq b_{\{1\}}+b_{\{2\}}+b_{\{3\}}+\ldots,b_{\{k\}},
$$
which can only be if $b_\emptyset=0$, which contradicts that $\psi_{\emptyset} \in \Supp(\beta)$.
\par The dimension vector $\beta=(b_\emptyset,0,0,b_{\{1,2\}},\ldots,0)$ with $b_\emptyset = b_{\{1,2\}}$ does fulfil condition \ref{conditionsimple}, but then we are in the extended Dynkin case of theorem \ref{th:simpledim}, which is only simple if $b_\emptyset=b_{\{1,2\}}=1$.

\par In all other cases, if $\beta$ is a simple dimension vector, then we find for all $A \in 2^N$
$$
\chi_{Q'_n}(\psi_A,\beta)=\chi_{Q'_n}(\beta,\psi_A) = \sum_{B \in 2^N} b_B(1-|A\Delta B|)\leq 0,
$$
which is equivalent to condition \ref{conditionsimple}.
\end{proof}
This should now be translated to a condition on the dimension vectors for $Q_n$. We will assume that
$ \alpha = (a_i^+,a_i^-)_{i=1}^n$ with $a_i^++a_i^-=m$ and
\begin{align}
a_1^+\geq a_2^+ \geq \ldots a_n^+ \geq a_n^- \geq a_{n-1}^- \geq \ldots \geq a_1^-.
\label{al:condalpha}
\end{align}
\begin{proposition}
The following semisimple representation is an element of $\rep_\alpha \Z_2^{*n}$:
\begin{align*}
M_\alpha &=\psi_\emptyset^{a_n^+} \oplus \psi_{\{n\}}^{a_{n-1}^+-a_n^+} \oplus \ldots \oplus\psi_{\{i+1,\ldots,n\}}^{a_i^+- a_{i+1}^+}\oplus \ldots \oplus \psi_{\{2,\ldots,n\}}^{a_1^+-a_2^+} \oplus \psi_N^{a_1^-}\\
         &=\psi_\emptyset^{a_n^+}\oplus \psi_{\{1\}^c}^{a_{2}^--a_1^-} \oplus \ldots \oplus\psi_{\{1,\ldots,i\}^c}^{a_{i+1}^--a_i^-}\oplus \ldots \oplus \psi_{\{1,\ldots,n-1\}^c}^{a_n^--a_{n-1}^-} \oplus \psi_N^{a_1^-}.
\end{align*}
\end{proposition}
\begin{proof}
Let $\chi_\alpha$ be the character function associated to $M_\alpha$, then we find
\begin{align*}
\chi_\alpha(e_i) &= a_n^+ - a_1^- - \sum_{j=1}^{i-1}(a_{j+1}^--a_j^-) + \sum_{j=i}^{n-1} (a_{j+1}^--a_j^-)\\
				 &= a_n^+ - a_1^- - (a_i^- - a_1^-) + (a_n^- - a_i^-)\\
				 &= m - 2a_i^-\\
				 &= a_i^+ - a_i^-.
\end{align*}
Together with the condition $a_i^+ + a_i^- = m$ we find that $M_\alpha \in \rep_\alpha \Z_2^{*n}$.
\end{proof}
\begin{proposition}
If $\alpha$ is a dimension vector that fulfils condition \ref{al:condalpha}, then $\alpha$ is a simple dimension vector if and only if 
$$
\sum_{i=1}^n a_i^+ \leq m(n-1),
$$
unless $\alpha = (2k,0;2k,0;\ldots;2k,0;2k;0;k,k;k,k)$, $m=2k$ and $k \neq 1$.
\end{proposition}
\begin{proof}
In order for $\alpha$ to be a simple dimension vector, we need that the dimension vector $\beta = (b_B)_{B \in 2^N}$ with $b_{\{i+1,\ldots,n\}} = a_i^+-a_{i+1}^+$ for $2\leq i \leq n-1$, $b_\emptyset = a_n^+$, $b_N = a_1^-$ and the rest equal to 0 fulfils condition \ref{conditionsimple}. Condition \ref{conditionsimple} then becomes for each $A \in 2^N$
\small
\begin{align}
|\beta| = m &\leq \sum_{B \in 2^N} b_B |A \Delta B|= a_n^+ |A| + a_1^- |A^c|+\sum_{i=1}^{n-1} (a_i^+-a_{i+1}^+)|A \Delta \{i+1,\ldots,n\}|.
\label{inequalitydim}
\end{align}
\normalsize
This last summation is equal to
\begin{align*}
m|A^c|+\sum_{i=1}^n (|A \Delta \{i+1,\ldots,n\}|-|A \Delta \{i,\ldots,n\}|)a_i^+.
\end{align*}
Define for each $A \in 2^N, i \in N$
$$
\delta^i_A = \begin{cases}
1 &\text{if } i \in A\\
-1 &\text{if } i \not\in A
\end{cases}
$$
Then inequality \ref{inequalitydim} is equivalent to
$$
0 \leq m(|A^c|-1)+\sum_{i=1}^n \delta^i_A a_i^+ \Leftrightarrow \sum_{i=1}^n \delta^i_{A^c} a_i^+ \leq m(|A^c|-1),
$$
for each $A \in 2^N$, or, as $(-)^c$ is an involution of $2^N$,
$$
\sum_{i=1}^n \delta^i_{A} a_i^+ \leq m(|A|-1).
$$
Let $\mathcal{O}_k$ for $k\in N \cup \{0\}$ be the orbits in $2^N$ under the $S_n$-action. Then the value $m(|A|-1)$ is constant on $\mathcal{O}_k$. Due to condition \ref{al:condalpha}, for every $k \in N\cup \{0\}$, $\sum_{i=1}^n \delta^i_{A} a_i^+$ reaches its maximum on $\mathcal{O}_k$ in the set $\{1,2,\ldots,k\}$.
\begin{itemize}
\item If $A = \emptyset$, then the condition becomes
$$
\sum_{i=1}^n - a_i^+ \leq -m \Leftrightarrow 0 \leq -m + \sum_{i=1}^n( m - a_i^-)\Leftrightarrow \sum_{i=1}^n  a_i^-\leq m(n-1).
$$
\item If $A = \{1,\ldots,k\}$, then this condition is equivalent to
\begin{align*}
\sum_{i=1}^k a_i^+ \leq m(k-1) + \sum_{i=k+1}^n a_i^+ &\Leftrightarrow \sum_{i=1}^k a_i^+ \leq m(k-1) + \sum_{i=k+1}^n (m-a_i^-) \\&\Leftrightarrow  \sum_{i=1}^k  a_i^+ + \sum_{i=k+1}^n  a_i^-\leq m(n-1).
\end{align*}
\end{itemize}
In particular, as $a_i^- \leq a_i^+$, it follows that 
$$
 \sum_{i=1}^n  a_i^+ \leq m(n-1)
$$
is a necessary (for $A=N$) and sufficient condition.
\par The dimension vectors of $Q_n$ that are excluded come from the dimension vectors of $Q'_n$ of proposition \ref{prop:simplocal} that are excluded.
\end{proof}
\begin{theorem}
Let $\alpha = (a_1^+,a_1^-;a_2^+,a_2^-;\ldots;a_n^+,a_n^-)$ be a dimension vector of $Q_n$, then $\alpha$ is a simple dimension vector if and only if
$$
\sum_{i=1}^n \max\{a_i^+,a_i^-\} \leq m(n-1),
$$
unless $\alpha$ lies in the $B_n$-orbit of $(2k,0;2k,0;\ldots;2k,0;2k;0;k,k;k,k)$ with $m=2k$ and $k \neq 1$.
\label{th:simpledimgen}
\end{theorem}
\begin{example}
The symmetric group $S_{n+1}$ is a quotient of $\Z_2^{*n}$ by taking the generators $(i,i+1), 1\leq i \leq n$ of $S_{n+1}$. As is well known, there exists a $n$-dimensional simple representation $V_{n}$ of $S_{n+1}$ such that $V_n \oplus \psi_{\emptyset}$ is the permutation representation of $S_{n+1}$. The representation $V_{n}$ lies in the component $\rep_\alpha \Z_2^{*n}$ with $\alpha = (n-1,1)_{i=1}^n$. For each $n$, we see that the condition of theorem \ref{th:simpledimgen} for this component is true, we even have equality:
$$
\sum_{i=1}^n \max\{a_i^+,a_i^-\} = n(n-1).
$$
However, if we look at the $n+1$-dimensional permutation representation $V_n \oplus \psi_\emptyset$, then we get $\alpha = (n,1)_{i=1}^n$ and therefore
$$
\sum_{i=1}^n \max\{a_i^+,a_i^-\} = n^2 > (n+1)(n-1)=n^2-1.
$$
So $\alpha = (n,1)_{i=1}^n$ is not a simple dimension vector.
\end{example}
\begin{proposition}
If $\alpha$ is a simple dimension vector, then
$$
\dim \iss_\alpha \Z_2^{*n}=2 (\sum_{i=1}^n a_i^+a_i^-)-(m^2-1).
$$
\label{dimsimple}
\end{proposition}
\begin{proof}
If $\alpha$ is a simple dimension vector, then 
\begin{align*}
\dim \iss_\alpha \Z_2^{*n} &= (n-1)\dim \GL_m(\C) - (\dim \left(\prod_{i=1}^n \GL_{a_i^+}(\C) \times \GL_{a_i^-}(\C)\right)-1)\\
						   &= m^2(n-1) - (-1+\sum_{i=1}^n (a_i^+)^2 +(a_i^-)^2)\\
                           &= m^2(n-1) - (-1+nm^2 - 2\sum_{i=1}^n a_i^+a_i^-)\\
                           &= 2 (\sum_{i=1}^n a_i^+a_i^-)-(m^2-1).
\end{align*}
\end{proof}
\subsection{Smooth dimension vectors}
The next question we can ask ourselves is: when is $\iss_\alpha \Z_2^{*n} =  \rep_\alpha \Z_2^{*n}/\GL_\alpha(\C)$ a smooth variety? For this to investigate, we need to calculate local quivers in sums of one-dimensional representations and see if $\iss_m \Z_2^{*n}$ is smooth in these points.
\begin{theorem}\cite[Theorem 5.22]{LBBook}
Let $(Q,\alpha)$ be a symmetric quiver setting with $Q$ connected and containing no loops. Then $\iss_\alpha Q$ is smooth if and only if
\begin{itemize}
\item $Q$ is tree-like, which means that if we draw an edge between 2 vertices if there exists an arrow between them, we get a tree.
\item In every branching vertex, the dimension is equal to 1.
\item The quiver is built out of the following blocks:
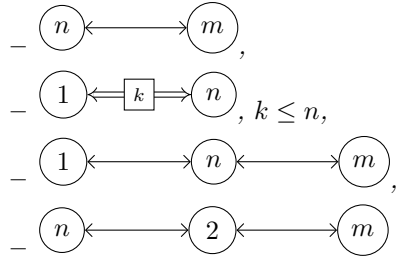
\begin{center}
\begin{itemize}
\item \begin{tikzpicture}
\node[vertice,circle] (a) at ( -1, 0) {$n$};
\node[vertice,circle] (b) at ( 1, 0) {$m$};
\path[->,font=\scriptsize,>=angle 90]
    (a) edge[<->]  (b)
    ;
\end{tikzpicture},
\item \begin{tikzpicture}[
    implies/.style={double,double equal sign distance,implies-implies},
    dot/.style={shape=circle,fill=black,minimum size=2pt,
                inner sep=0pt,outer sep=2pt}]
\node[vertice,circle] (a) at ( -1, 0) {$1$};
\node[vertice,circle] (b) at ( 1, 0) {$n$};
\tikzset{every node/.style={fill=white},rectangle}
\draw[->,font=\scriptsize]
    (a) edge[implies] node[vertice,rectangle,anchor=center]{$k$}(b)
    ;
\end{tikzpicture}, $k\leq n$,
\item \begin{tikzpicture}
\node[vertice,circle] (a) at ( -2, 0) {$1$};
\node[vertice,circle] (b) at ( 2, 0) {$m$};
\node[vertice,circle] (c) at ( 0, 0) {$n$};
\path[->,font=\scriptsize,>=angle 90]
    (a) edge[<->]  (c)
    (b) edge[<->]  (c)
    ;
\end{tikzpicture},
\item \begin{tikzpicture}
\node[vertice,circle] (a) at ( -2, 0) {$n$};
\node[vertice,circle] (b) at ( 2, 0) {$m$};
\node[vertice,circle] (c) at ( 0, 0) {$2$};
\path[->,font=\scriptsize,>=angle 90]
    (a) edge[<->]  (c)
    (b) edge[<->]  (c)
    ;
\end{tikzpicture}
\end{itemize}
\end{center} 
\end{itemize}
\label{th:symmetric}
\end{theorem}
Let us first find the possible full subquivers which are tree-like in $Q'_n$.
\begin{proposition}
There are $2n-1$ types of connected tree-like full subquivers in $Q'_n$ that can be subdivided in 4 cases, corresponding to
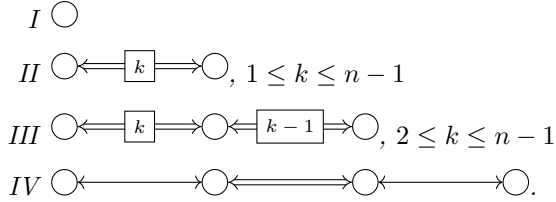
\begin{center}
\begin{enumerate}[I]
\item \begin{tikzpicture}
\node[vertice,circle] (a) at ( 0, 0) {};
\end{tikzpicture}
\item \begin{tikzpicture}[
    implies/.style={double,double equal sign distance,implies-implies},
    dot/.style={shape=circle,fill=black,minimum size=2pt,
                inner sep=0pt,outer sep=2pt}]
\node[vertice,circle] (a) at ( -1, 0) {};
\node[vertice,circle] (b) at (  1, 0) {};
\tikzset{every node/.style={fill=white},rectangle}
\draw[->,font=\scriptsize]
    (a) edge[implies] node[vertice,rectangle,anchor=center]{$k$}(b)
    ;
\end{tikzpicture}, $1\leq k \leq n-1$
\item \begin{tikzpicture}[
    implies/.style={double,double equal sign distance,implies-implies},
    dot/.style={shape=circle,fill=black,minimum size=2pt,
                inner sep=0pt,outer sep=2pt},scale=2]
\node[vertice,circle] (a) at ( -1, 0) {};
\node[vertice,circle] (b) at (  0, 0) {};
\node[vertice,circle] (c) at (  1, 0) {};
\tikzset{every node/.style={fill=white},rectangle}
\draw[->,font=\scriptsize]
    (a) edge[implies] node[vertice,rectangle,anchor=center]{$k$}(b)
    (c) edge[implies] node[vertice,rectangle,anchor=center]{$k-1$}(b)
    ;
\end{tikzpicture}, $2\leq k \leq n-1$
\item \begin{tikzpicture}[
    implies/.style={double,double equal sign distance,implies-implies},
    dot/.style={shape=circle,fill=black,minimum size=2pt,
                inner sep=0pt,outer sep=2pt},scale=2]
\node[vertice,circle] (a) at ( -1, 0) {};
\node[vertice,circle] (b) at (  0, 0) {};
\node[vertice,circle] (c) at (  1, 0) {};
\node[vertice,circle] (d) at (  2, 0) {};
\tikzset{every node/.style={fill=white},rectangle}
\draw[->,font=\scriptsize]
    (a) edge[<->] (b)
    (c) edge[implies] (b)
    (c) edge[<->] (d)  
    ;
\end{tikzpicture}.
\end{enumerate}
\end{center}
\label{prop:treelike}
\end{proposition}
\begin{proof}
The connected tree-like full subquivers that contain $\leq 2$ vertices are trivially the ones from the first 2 cases, so we may assume that there are at least 3 vertices. Using the $B_n$-action, we may also assume that $\psi_\emptyset$ is a vertex which is connected to only one other vertex. Assume that the third vertex is not connected to $\psi_\emptyset$, then again using the $B_n$-action we may assume that this third vertex is $\psi_{\{1\}}$. Consider now the second vertex $\psi_A$ for some $A \in 2^N$. This vertex should be connected to $\psi_\emptyset$, so $|A| \geq 2$. There are now two options:
\begin{itemize}
\item There is no fourth vertex: then $A$ is an element of $2^N$ with $|A| \geq 2$ and $|A \Delta \{1\}|\geq 2$. There are two possibilities:
\begin{itemize}
\item $1 \in A$: then the number of arrows from $\psi_A$ to $\psi_{\{1\}}$ is one less than the number of arrows from $\psi_A$ to $\psi_\emptyset$, or
\item $1 \not\in A$: then the number of arrows from $\psi_A$ to $\psi_{\{1\}}$ is one more than the number of arrows from $\psi_A$ to $\psi_\emptyset$.
\end{itemize}
These two cases both correspond to the third option of the proposition. The number of arrows follow from the description of $Q'_n$ from corollary \ref{cor:onequiver}.
\item There is a fourth vertex: there are two options to consider:
\begin{itemize}
\item This fourth vertex is connected to $\psi_A$: as this vertex $\psi_B$ is not connected to $\psi_\emptyset$, one has $B = \{j\}$ for some $j \neq 1$, but then $\psi_{\{1\}}$ is connected to $\psi_B$. Therefore, this can not happen.
\item this fourth vertex corresponds to a set $B$ such that $|B| = 1$ and $|B \Delta \{1\}|\geq 2$. By the $B_n$-action we may assume that $B = \{2\}$, so $|B \Delta \{1\}|= 2$. Consequently, $A$ is a set so that $|A|\geq 2$, $|A \Delta \{1\}|\geq 2$ and $|A \Delta \{2\}| = 1$. This means that $2 \in A$, $|A| = 2$, $1 \not\in A$. By the $B_n$-action we may assume that $A = \{2,3\}$.
\end{itemize} 
\par Suppose now that there is a fifth vertex, as this vertex is not connected to $\psi_\emptyset$ this corresponds to $\{k\}$ for some $3\leq k \leq n$. But therefore, this vertex is always connected to $\psi_{\{1\}}$, which is impossible. This means that there can be no fifth vertex, so we are in the fourth case of the proposition. The number of arrows are again determined by corollary \ref{cor:onequiver}.
\end{itemize} 
\end{proof}
While this means that there are indeed dimension vectors for $Q'_n$ such that $\iss_\beta Q'_n$ is smooth and consequently, its image in $\iss_m \Z_2^{*n}$ is smooth, this does not mean that the component in which its image lies is smooth. In fact,
\begin{theorem}
If $\iss_\alpha \Z_2^{*n}$ is smooth, then $\alpha$ lies in the $B_n$-orbit of the dimension vector $(a,b;c,d;m,0;\ldots;m,0)$ for some $a,b,c,d,m\in \N$ and $a+b = m = c+d$.
\end{theorem}
\begin{proof}
Assume first that $\alpha = (a,b;c,d;m,0;\ldots;m,0)$ with $a+b=c+d=m$. Then every direct sum of one-dimensional simple representations in $\rep_\alpha \Z_2^{*n}$ is determined by a four-tuple $(x,y,z,v) \in \N^4$ such that $x+z=a$, $x+y = c$, $y+v=b$, $z+v=d$. The corresponding semisimple representation is then
$$
M_\alpha = \psi_\emptyset^x \oplus \psi_{\{1\}}^y \oplus \psi_{\{2\}}^z \oplus \psi_{\{1,2\}}^v.
$$
The local quiver setting in this representation is 
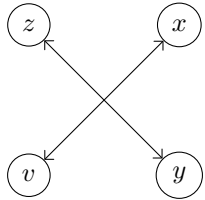
\begin{center}
\begin{tikzpicture}[scale=1]
   \node[vertice,circle] (a) at ( 1, 1) {$x$};
   \node[vertice,circle] (b) at ( 1, -1) {$y$};
   \node[vertice,circle] (c) at (  -1, 1) {$z$};
   \node[vertice,circle] (d) at (  -1, -1) {$v$};
   \path[->,font=\scriptsize,>=angle 90]
    (a) edge[<->] (d)
    (b) edge[<->] (c)
    ;
\end{tikzpicture}
\end{center}
This is a smooth quiver setting by theorem \ref{th:symmetric}. It is possible for some vertex to have dimension $0$. Therefore, for each $\alpha'$ in the $B_n$-orbit of $\alpha$, $\iss_\alpha \Z_2^{*n}$ is smooth.
\par Let $\beta$ now be a dimension vector of $Q'_n$, with $\Supp(\beta)$ one of the tree-like full subquivers of proposition \ref{prop:treelike}, excluding type I. For type II and III, we will assume that $k\geq 2$, otherwise we will be in the previous smooth case.
\begin{itemize}
\item $\Supp(\beta)$ is of type II: we may assume that $b_\emptyset \geq k$, $b_{1,\ldots,k+1}=1$ and for all other elements of $2^N$, $b_A = 0$. In the abelian semigroup $S(n)$ with generators $\alpha_A$ (corresponding to the one-dimensional representation $\psi_A$) for $A \in 2^N$, we have the equality
$$
m \alpha_\emptyset + \alpha_{\{1,\ldots,k+1\}} = (m-k) \alpha_\emptyset + \sum_{i=1}^{k+1} \alpha_{\{i\}}.
$$
Thus $\psi_\emptyset^m \oplus \psi_{\{1,\ldots,k+1\}}$ and $\psi_\emptyset^{m-k} \oplus \bigoplus_{i=1}^{k+1} \psi_{\{i\}}$ belong to the same component $\rep_\alpha\Z_2^{*n}$. The local quiver in $\psi_\emptyset^{m-k} \oplus \bigoplus_{i=1}^{k+1} \psi_{\{i\}}$ is not tree-like, thus $\iss_\alpha\Z_2^{*n}$ is not smooth.
\item $\Supp(\beta)$ is of type III: same proof as in type II, by disregarding the third vertex and focusing on the first 2 vertices.
\item $\Supp(\beta)$ is of type IV: we may assume that $b_{\emptyset}$, $b_{\{1\}}$, $b_{\{2\}}$ and $b_{\{2,3\}}$ are nonzero. As $\psi_{\{1\}}$ and $\psi_{\{2,3\}}$ are branching vertices, if $\iss_\beta Q_n'$ was smooth, then $b_{\{1\}} = b_{\{2,3\}} = 1$, contradicting the fact that one of these two values should be $\geq 2$.
\end{itemize}
\end{proof}
\begin{corollary}
There are no simple, smooth dimension vectors in $S(m)$ unless $m\leq 2$.
\end{corollary}

\section{Explicit computational results}
\subsection{Two-dimensional representations}
We can completely determine the components and their dimensions in $\rep_2\Z_2^{*n}$ and $\iss_2\Z_2^{*n}$ by hand.
\begin{theorem}
The representation variety $\rep_2\Z_2^{*n}$ consists of $3^n$ components. For $0 \leq k \leq n$, there are $E_{k,n} = 2^{n-k} \binom{n}{k}$ components of dimension $2k$, indexed by $A \in 2^N$, $|A|=k$ and $B \in 2^{N\setminus A}$.
\par Let $V_{A,B,n}$ be such a component with $|A|=k$. Then the GIT-quotient $W_{A,B,n}=V_{A,B,n}/\GL_2(\C)$ is of dimension $2k-3$, except if $k = 0,1$, in which case the dimension is 0.
\par The number of singularities in $W_{A,B,n}$ is equal to $2^{k-1}$ if $k\geq 3$. These singularities correspond to the sum of two one-dimensional representations $\psi_C,\psi_D$ with $|C\Delta D| = k$, with corresponding local quiver setting
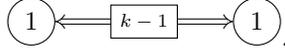
\begin{center}
\item \begin{tikzpicture}[
    implies/.style={double,double equal sign distance,implies-implies},
    dot/.style={shape=circle,fill=black,minimum size=2pt,
                inner sep=0pt,outer sep=2pt}]
\node[vertice,circle] (a) at ( -1.5, 0) {$1$};
\node[vertice,circle] (b) at (  1.5, 0) {$1$};
\tikzset{every node/.style={fill=white},rectangle}
\draw[->,font=\scriptsize]
    (a) edge[implies] node[vertice,rectangle,anchor=center]{$k-1$}(b)
    ;
\end{tikzpicture}. 
\end{center}
If $k =0,1,2$, there are no singularities.
\end{theorem}
\begin{proof}
In the matrix ring $\M_2(\C)$, denote with $\mathcal{O}_0,\mathcal{O}_1,\mathcal{O}_2$ the $\GL_2(\C)$-orbits of respectively
$$
\begin{bmatrix}
1 & 0 \\ 0 & -1
\end{bmatrix},
\begin{bmatrix}
1 & 0 \\ 0 & 1
\end{bmatrix},
\begin{bmatrix}
-1 & 0 \\ 0 & -1
\end{bmatrix}.
$$
Then any component of $\rep_2\Z_2^{*n}$ is a product of $n$ choices of $\mathcal{O}_0,\mathcal{O}_1$ and $\mathcal{O}_2$. This gives $3^n$ components. The dimension of such a component depends on the number of times $\mathcal{O}_0$ is chosen, for $\dim(\mathcal{O}_1)=\dim(\mathcal{O}_2)=0$. The dimension of $\mathcal{O}_0$ is equal to $2$, as can be easily seen using the formula
$$
\dim \GL_2(\C) = \dim \mathcal{O}_0 + \dim \Stab_{\GL_2(\C)}\begin{bmatrix}
1 & 0 \\ 0 & -1
\end{bmatrix}= \dim \mathcal{O}_0 + \dim \C^* \times \C^*.
$$
\par Let $s$ be the number of components of dimension $2k$. For such a component, one first chooses a subset $A \subset N$ consisting of the factors equal to $\mathcal{O}_0$ and then choosing a subset $B \subset N \setminus A$, consisting of the factors equal to $\mathcal{O}_2$. The other factors are equal to $\mathcal{O}_1$. This shows that $s = 2^{n-k} \binom{n}{k}$.
\par Choose a component $V_{A,B,n}$ with $|A|=k \geq 2$, let $i,j \in A$, $i \neq j$. Then there is a surjective restriction map
\begin{tikzcd}
V_{A,B,n} \arrow[r, two heads] & V_{\{1,2\},\emptyset,2}
\end{tikzcd}
by projecting on the $i$th and $j$th component. As $V_{\{1,2\},\emptyset,2}$ contains an open subset of simple representations, it follows that $V_{A,B,n}$ has an open subset of simple representations.
\par This means that on an open subset $U \subset W_{A,B,n}$,  $V_{A,B,n}$ is a principal $\PGL_2(\C)$-fibration in the \'etale topology, leading to $\dim W_{A,B,n} = 2k-3$.
\par From this it follows that the possible singularities of $\dim W_{A,B,n}$ lie in the semisimple representations. Semisimple representations are determined by two characters $\psi_C, \psi_D$. One of the requirements that $\psi_C \oplus \psi_D$ lies in $V_{A,B,n}$ is that $|C\Delta D| = k$, from which the claimed local quiver type follows. From this it also follows that $W_{A,B,n}$ does not have singularities if $k \leq 2$, as 
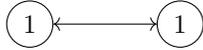
\begin{center}
\begin{tikzpicture}
\node[vertice,circle] (a) at ( -1, 0) {1};
\node[vertice,circle] (b) at ( 1, 0) {1};
\draw[->,font=\scriptsize]
    (a) edge[<->]  (b)
    ;
\end{tikzpicture}
\end{center}
is a smooth local quiver setting.
\par The condition that $\psi_C \oplus \psi_D \in V_{A,B,N}$ can be descibed as
$$
\forall j \in N:
\begin{cases}
\psi_C(e_j) = \pm 1 \text{ and } \psi_D(e_j) = \mp 1 & j \in A,\\
\psi_C(e_j) = \psi_D(e_j) = -1	                     & j \in B,\\
\psi_C(e_j) = \psi_D(e_j) =  1	                     & j \in N\setminus (A \cup B).
\end{cases}
$$
Consequently, there are $2^k$ possible semisimple representations $\psi_C \oplus \psi_D \in V_{A,B,n}$, but as $\psi_C \oplus \psi_D$ is conjugated to $\psi_D \oplus \psi_C$, it follows that there are $2^{k-1}$ singularities in $W_{A,B,n}$ if $k \geq 3$.
\end{proof}
In terms of the $n$-dimensional hypercube $2^N$, we see that the irreducible components of $\rep_2 \Z_2^{*n}$ are in bijection with the subcubes of $2^N$. Every component of $\rep_2 \Z_2^{*n}$ contains simple representations, except for the components corresponding to points and edges in $2^N$.
\begin{example}
The 2-dimensional representations of $\C(\Z_2^{*3})$ correspond to the representations of a Clifford algebra $\mathcal{C}$ over $\C[X_{1,2},X_{1,3},X_{2,3}]$. The associated quadratic form over $\C[X_{1,2},X_{1,3},X_{2,3}]$ is determined by the matrix
$$
\begin{bmatrix}
1 & X_{1,2} & X_{1,3}\\ X_{1,2} & 1 & X_{2,3} \\ X_{1,3} & X_{2,3} & 1
\end{bmatrix}.
$$
Let $\rho$ be the natural algebra map
$$
\begin{tikzcd}
\C(\Z_2^{*3})  \arrow["\rho", r]&\mathcal{C}
\end{tikzcd}.
$$
The center of $\mathcal{C}$ by $\cite{le1987trace}$ is equal to
$$
R=\C[X_{1,2},X_{1,3},X_{2,3},g]/(g^2-(1-(X_{1,2}^2+X_{1,3}^2+X_{2,3}^2)+2X_{1,2}X_{1,3}X_{2,3})),
$$
with $g$ a scalar multiple of $\rho(\sum_{\sigma \in S_3} \varepsilon(\sigma) x_{\sigma(1)}x_{\sigma(2)}x_{\sigma(3)})$.
\par We have $\Spec(R)=\rep_{(1,1,1,1,1,1)} \Z_2^{*3}/\GL_2(\C)$ and one has
$$\rep_{(1,1,1,1,1,1)} \Z_2^{*3} = \trep_2 \mathcal{C}_3.$$
As expected, $\Spec(R(3))$ has 4 singularities, in the points $$(1,1,1,0),(1,-1,-1,0),(-1,1,-1,0),(-1,-1,1,0),$$ which are all of conifold type ($3_{con}$ in the classification of \cite[proposition 5.13]{LBBook}). The 4 singularities form a $S_4= B_3/(\Z_2)$-orbit.
\end{example}
 \subsection{The dimension vector $(m-1,1)_{i=1}^n$}
An easy application of theorem \ref{prop:simplocal} is the following:
\begin{proposition}
The dimension vector $\alpha(n,m) = (m-1,1)_{i=1}^n$ is simple if and only if $m \leq n$. In this case, $\dim \iss_{\alpha(n,m)} \Z_2^{*n} = (m-1)(2n-m-1)$.
\label{prop:alphamsimpel}
\end{proposition}
\begin{proof}
Theorem \ref{th:simpledimgen} applied for $\alpha(n,m)$ becomes
$$
n(m-1) \leq m(n-1) \Leftrightarrow m \leq n.
$$
Consequently, proposition \ref{dimsimple} applied to $\alpha(n,m)$ becomes
$$
\dim \iss_{\alpha(n,m)} \Z_2^{*n} = 2n(m-1)-(m^2-1) = (m-1)(2n-m-1).
$$
\end{proof}
For $A \in 2^N$, we will set
$$
\alpha(n,m,A) = (a_i^+,a_i^-)_{i=1}^n, (a_i^+,a_i^-) =
\begin{cases}
(m-1,1) &\text{ if } i \in A,\\
(m,0)   &\text{ if } i\not\in A.
\end{cases}
$$
Consequently, $\alpha(n,m,N) = \alpha(n,m)$ and $\alpha(n,m,\emptyset) = (m,0)_{i=1}^{n}$. In addition, the vector $\alpha(n,m,A)$ determines the representations that have the subgroup $\langle e_j : j\not\in A \rangle \subset \Z_2^{*n}$ in its kernel. We will set $\Z_2^{*n}(A) = \langle e_j : j\in A\rangle \cong \Z_2^{*|A|}$, then 
$$
\rep_{\alpha(n,m,A)} \Z_2^{*n} \cong \rep_{\alpha(|A|,m)}\Z_2^{*|A|}
$$
by way of the restriction morphism $V \mapsto V|_{\Z_2^{*n}(A)}$.
\par We will classify all local quiver settings in $\dim \iss_{\alpha(n,m)}\Z_2^{*n}$ for all $m \leq n$, which will be a generalization of theorem \ref{theorem:mainintro}.
\begin{theorem}
The possible local quiver settings in $\dim \iss_{\alpha(n,m)}\Z_2^{*n}$ for $m \leq n$ are determined by the following data:
\begin{itemize}
\item a partition $\mathbf{A}=\{A_i\}_{i=1}^l$ of $N$, and
\item a positive $l$-tuple $\mathbf{k}=(k_i)_{i=1}^l \in \N^l$ satisfying $1\leq k_i \leq |A_i|$ for all $1\leq i \leq l$ and $|\mathbf{k}| = \sum_{i=1}^l k_i \leq m$.
\end{itemize}
A local quiver setting $(\mathbf{A},\mathbf{k})$ degenerates to $(\mathbf{A'},\mathbf{k'})$ if and only if
\begin{itemize}
\item $\mathbf{A'}$ is a refinement of $\mathbf{A}$, and
\item for all $1\leq i\leq l$, if $A_i = \bigsqcup_{j=1}^{l_i} A'_{i_j}$, then $k_i \geq \sum_{j=1}^{l_i} k'_{i_j}$. 
\end{itemize}
The local quiver setting corresponding to $(\mathbf{A},\mathbf{k})$ is the following:
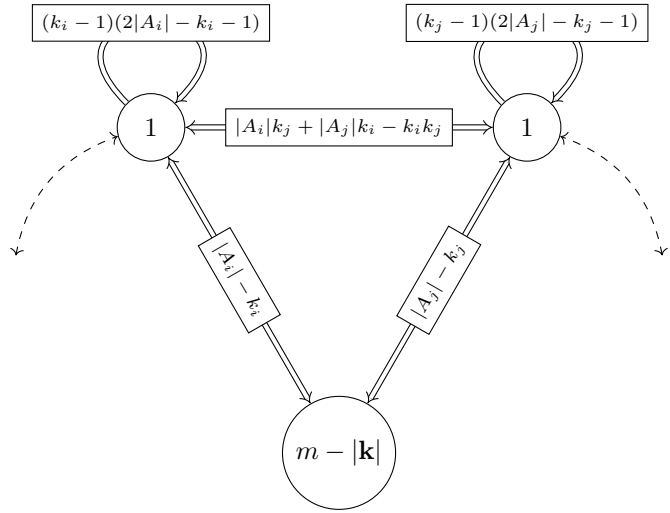
\begin{center} 
\begin{tikzpicture}[
    implies/.style={double,double equal sign distance,implies-implies},
    loopy/.style={double,double equal sign distance,-implies},
    dot/.style={shape=circle,fill=black,minimum size=2pt,
    state/.style={circle, draw, minimum size=2cm},inner sep=0pt,outer sep=2pt},scale=5]
\node[state] (Aj) at ( {cos(60)}, {sin(60)}) {$1$};
\node[state] (lege) at (  0, 0) {$m-|\mathbf{k}|$};
\node[state] (Ai) at ( {cos(120)}, {sin(120)}) {$1$};
\node (emptyone) at ({cos(150)},{sin(150)}) {};
\node (emptyoneorigin) at ({cos(30)},{sin(30)}) {};
\tikzset{every node/.style={fill=white},rectangle}
\draw[->,font=\scriptsize]
    (Ai) edge[implies]node[vertice,rectangle,anchor=center]{$|A_i|k_j+|A_j|k_i-k_ik_j$} (Aj)
    (Ai) edge[loopy,loop above,out=135,in=45,looseness=8]node[vertice,rectangle,anchor=center]{$(k_i-1)(2|A_i|-k_i-1)$} (Ai)        
    (Aj) edge[loopy,loop above,out=135,in=45,looseness=8]node[vertice,rectangle,anchor=center]{$(k_j-1)(2|A_j|-k_j-1)$} (Aj)        
    (Ai) edge[implies]node[vertice,rectangle,rotate=-60]{$|A_i|-k_i$} (lege) 
    (Aj) edge[implies]node[vertice,rectangle,rotate=60]{$|A_j|-k_j$} (lege) 
    (emptyone) edge[dashed,<->,bend left] (Ai)       
    (emptyoneorigin) edge[dashed,<->,bend right] (Aj) 
    ;
\end{tikzpicture}
\end{center}
\label{theorem:main}
\end{theorem}
\begin{proof}
We will first show the correspondence between the possible local quiver settings and partitions $\mathbf{A}=\{A_i\}_{i=1}^l$ and $l$-tuples $\mathbf{k}=(k_i)_{i=1}^l$ satisfying $1\leq k_i \leq |A_i|$ for all $1\leq i \leq l$ and $|\mathbf{k}| = \sum_{i=1}^l k_i \leq m$.
\par Let $M=\bigoplus_{i=0}^l V_i^{e_i}$ be a semisimple decomposition of $M \in \rep_{\alpha(n,m)}\Z_2^{*n}$ with $V_0 = \psi_\emptyset$, $e_i \geq 1$ for $1\leq i \leq l$ and $e_0 \geq 0$. Set $k_i = \dim V_i$ for $1\leq i \leq l$. Then we have
$$
e_i \neq 1 \Rightarrow V_i \neq \psi_\emptyset,
$$
for if $V_i \neq \psi_\emptyset$ and $e_i \geq 2$, then $M \not\in \rep_{\alpha(n,m)}\Z_2^{*n}$, as then there would exist an element $j \in N$ such that $(\alpha(n,m))_j^- \geq 2$, which is a contradiction.
\par For $i \neq 0$, let $V_i \in \rep_{\beta(i)}\Z_2^{*n}$ and set
$$
A_i = \{j \in \N : (\beta(i))_j^- = 1\} = \{j \in \N : (\beta(i))_j^- \neq 0\}.
$$
Then the sets $A_1,\ldots,A_l$ form a partition of $N$ with $A_i \neq \emptyset$ for each $1\leq i \leq l$. In addition, from the condition that $(\beta(i))^+_j + (\beta(i))^-_j = k_i$ for all $j \in A_i$, it follows that
$$
\beta(i) = \alpha(n,k_i,A_i)
$$
By proposition \ref{prop:alphamsimpel}, this implies that $1\leq k_i \leq |A_i|$. But then, $e_0 = m - \sum_{i=1}^l k_i$, proving that $|\mathbf{k}|\leq  m$. 
\par We have associated a partition $\mathbf{A} = \{A_i\}_{i=1}^l$ of $N$ and a positive $l$-tuple $\mathbf{k}=(k_i)_{i=1}^l$ satisfying $1\leq k_i \leq |A_i|$ for all $1\leq i \leq l$ and $|\mathbf{k}|\leq  m$ to a semisimple representation in $\rep_{\alpha(n,m)} \Z_2^{*n}$.
\par The other way around, to a partition $\mathbf{A} = \{A_i\}_{i=1}^l$ of $N$ and a positive $l$-tuple $\mathbf{k}=(k_i)_{i=1}^l$ satisfying $1\leq k_i \leq |A_i|$ for all $1\leq i \leq l$ and $|\mathbf{k}|\leq  m$, we associate a semisimple representation $M \in \rep_{\alpha(n,m)}\Z_2^{*n}$ by choosing for each $1\leq i \leq l$ a simple representation $V_i \in \rep_{\alpha(n,k_i,A_i)}\Z_2^{*n}$, which we may do by proposition \ref{prop:alphamsimpel}, and taking $M_{\mathbf{A},\mathbf{k}} = \psi_\emptyset^{m-|\mathbf{k}|}\oplus\bigoplus_{i=1}^l V_l$. As the local quiver setting of a semisimple module depends purely on the dimension vectors of its semisimple components, we see that the local quiver setting does not depend on the chosen simple representations.
\par Regarding the degeneration statement, we will show that a local quiver setting $(\mathbf{A},\mathbf{k})$  degenerates to a quiver setting $(\mathbf{A}',\mathbf{k}')$ in the following two possibilities: 
\begin{itemize}
\item $\mathbf{A'} = \mathbf{A}$ and $k'_l+1 = k_l$, $k'_i = k_i$ for $1\leq i \leq l-1$,
\item $A_{l}' \sqcup A_{l+1}' = A_l$, $A_{i}' = A_i$ and $k'_{l}+k'_{l+1}=k_1$, $k'_{i}=k_i$ for $1\leq i \leq l-1$.
\end{itemize}
This will be sufficient, for degeneration is a transitive relation and a partition $\mathbf{A}=\{A_i\}_{i=1}^l$ is determined up to $S_l$-action, that is, $\mathbf{A}=\{A_i\}_{i=1}^l = \{A_{\sigma(i)}\}_{i=1}^l$ for each $\sigma \in S_l$. Let $M_{\mathbf{A'},\mathbf{k'}}$ be a semisimple representation of type $(\mathbf{A'},\mathbf{k'})$ in both cases.
\par If  $\mathbf{A'} = \mathbf{A}$ and $k'_l+1 = k_l$, then there exists a 1-dimensional family $\mathcal{C}$ of simples $(V_l(z))_{z \in \mathcal{C}} \in \rep_{\alpha(n,k_l,A_l)} \Z_2^{*n}$ that degenerates to $\psi_\emptyset \oplus V'_l$ with $V'_l \in \rep_{\alpha(n,k'_l,A'_l)} \Z_2^{*n}$.
Consequently, $M_{\mathbf{A'},\mathbf{k'}}$ is a degeneration of the family $(M_{\mathbf{A'},\mathbf{k'}}/(\psi_\emptyset\oplus V'_l) \oplus V_l(z))_{z \in \mathcal{C}}$, which are all of type $(\mathbf{A},\mathbf{k})$, proving the first possibility.
\par For the second possibility, there exists a 1-dimensional family $\mathcal{C}$ of simple representations $(V_l(z))_{z\in \mathcal{C}} \in \rep_{\alpha(n,k_l,A_l)} \Z_2^{*n}$ that degenerates to $V'_{l} \oplus V'_{l+1}$ with $V'_l \in \rep_{\alpha(n,k'_l,A'_l)} \Z_2^{*n}$ and $V'_{l+1} \in \rep_{\alpha(n,k'_{l+1},A'_{l+1})} \Z_2^{*n}$.
Consequently,  $(M_{\mathbf{A'},\mathbf{k'}}/(V'_l\oplus V'_{l+1}) \oplus V_l(z))_{z \in \mathcal{C}} $ is a 1-dimensional family of semisimple representations of type $(\mathbf{A},\mathbf{k})$ that degenerates to $M_{\mathbf{A'},\mathbf{k'}}$, finishing the second possibility.
\par Conversely, assume that the local quiver setting $(\mathbf{A}',\mathbf{k}')$ is a degeneration of $(\mathbf{A},\mathbf{k})$. A 1-dimensional family $\mathcal{C}$ of simple representations $(V_i(z))_{z \in \mathcal{C}} \in \rep_{\alpha(n,k_i,A_i)}\Z_2^{*n}$ can only degenerate to a sum of simples $\oplus_{j=1}^{l_i} V_{i_j}$ with $V_{i_j} \in \rep_{\alpha(n,k'_{i_j},A'_{i_j})}\Z_2^{*n}$ if $\sqcup_{j=1}^{l_i} A'_{i_j} = A_i$, for
$$
\psi_\emptyset^{k_i -\sum_{j=1}^{l_i} k_{i_j}}\oplus \bigoplus_{j=1}^{l_i}V_{i_j} \in \rep_{\alpha(n,k_i,A_i)}\Z_2^{*n}.
$$
This also shows that $\sum_{j=1}^{l_i} k_{i_j} \leq k_i$ for each $1\leq i \leq l$, finishing the proof of the degeneration statement.
\par To describe the local quiver setting in $M_{\mathbf{A},\mathbf{k}} = \psi_\emptyset^{m-|\mathbf{k}|}\oplus \bigoplus_{i=1}^l V_i$, we have
\begin{itemize}
\item the number of loops in the vertex corresponding to $V_i$ is equal to 
\begin{align*}
\dim \iss_{\alpha(n,k_i,A_i)} \Z_2^{*n}&= \dim \iss_{\alpha(|A_i|,k_i)} \Z_2^{*|A_i|}\\
									   &= (k_i-1)(2|A_i|-k_i-1)\\
									   &=k_i(|A_i|-k_i)+ (|A_i|-k_i)k_i + k_i^2 - 2|A_i| + 1.
\end{align*}
by proposition \ref{prop:alphamsimpel},
\item the number of arrows from the vertex corresponding to $V_i$ to the vertex corresponding to $V_j$ is (by the $S_n$ action on $\rep_{\alpha(n,m)} \Z_2^{*n}$) equal to
\footnotesize 
\begin{align*}
\dim\Ext^1_{\Z_2^{*n}}(V_i,V_j)&=-\chi_{Q_n}(\alpha(n,k_i,A_i),\alpha(n,k_j,A_j)) \\
               &= -\chi_{Q_n}(\alpha(n,k_i,\{1,\ldots,|A_i|\}),\alpha(n,k_j,\{|A_i|+1,\ldots,|A_i|+|A_j|\}))\\
               &=-\alpha(n,k_i,\{1,\ldots,|A_i|\})\mathcal{M}_n \alpha(n,k_j,\{|A_i|+1,\ldots,|A_i|+|A_j|\})^{t}\\
               &=-((k_i-1)k_j - (|A_i|-1)k_j - |A_j|k_i)\\
               &=|A_i|k_j + |A_j|k_i-k_ik_j\\
               &=k_i(|A_j|-k_j)+(|A_i|-k_i)k_j  + k_i k_j,
\end{align*}
\normalsize
\item the number of arrows from $\psi_\emptyset$ to the vertex corresponding to $V_i$ is
\begin{align*}
\dim\Ext^1_{\Z_2^{*n}}(\psi_\emptyset,V_i)&=-\chi_{Q_n}(\alpha(n,1,\emptyset),\alpha(n,k_i,A_i)) \\
							  &=|A_i|-k_i.
\end{align*}
The dimension vector is obvious for this local quiver setting. As the local quiver needs to be symmetric by the symmetry of $M_n$, this finishes the proof of the theorem.
\end{itemize}
\end{proof}
A consequence of theorem \ref{theorem:main} is the following:
\begin{corollary}
The Euler matrix of the local quiver setting $(\mathbf{A},\mathbf{k})$ for $\alpha(n,m)$ is equal to
$$
\begin{bmatrix}
-(\mathbf{k}^t\mathbf{v} + \mathbf{v}^t\mathbf{k} + \mathbf{k}^t \mathbf{k} - 2\diag(|A_1|,\ldots, |A_l|)) & -\mathbf{v}^t \\
-\mathbf{v} & 1 
\end{bmatrix},
$$
with $\mathbf{v} = (|A_i|-k_i)_{i=1}^l$, except if $\mathbf{v} = 0$, in which case the Euler matrix becomes
$$
-(\mathbf{k}^t \mathbf{k} - 2\diag(|A_1|,\ldots, |A_l|)) 
$$
\end{corollary}
It is well known that partitions of $N$ up to the action of $S_n$ correspond to Young diagrams. If we want to calculate the number of local quiver settings up to permutation in $\iss_{\alpha(n,n)} \Z_2^{*n}$, we get
\begin{proposition}
For each Young diagram given by the data $(\lambda_i,\mu_i)_{i=1}^l$ (so $\lambda_i > \lambda_{i+1}$ and there are $\mu_i$ rows with $\lambda_i$ boxes), there are $\prod_{i=1}^l \multiset{\lambda_i}{\mu_i}$ distinct local quiver settings in $\iss_{\alpha(n,n)} \Z_2^{*n}$.
\end{proposition}
\begin{proof}
For each rectangle subblock consisting of $\mu_i \lambda_i$ blocks, the possible local quiver settings depend on natural numbers $k_1, \ldots , k_{\mu_i}$ with $1\leq k_j \leq \lambda_i$ for each $1\leq j \leq \mu_i$. Up to permutation, we may assume that
$$
\lambda_i \geq k_1 \geq k_2 \geq \ldots \geq k_{\mu_i} \geq 1,
$$
which means that we have to choose $\mu_i$ elements in $\{1,\ldots,\lambda_i\}$, with repetition possible. This leads to the claimed multiset coefficient.
\end{proof}
We can also classify the smooth points in $\iss_{\alpha(n,m)} \Z_2^{*n}$ by their local quiver setting.
\begin{theorem}
A point $p \in \iss_{\alpha(n,m)} \Z_2^{*n}$ is smooth if and only if one of the following statements is true:
\begin{itemize}
\item The semisimple representation lying over $p$ is simple,
\item $n=m$ and the local quiver setting $(\mathbf{A},\mathbf{k})$ in $p$ has the property $\mathbf{A} = \{N\}$, or
\item $n=m=2$.
\end{itemize}
\end{theorem}
\begin{proof}
If the semisimple representation lying over $p$ is simple, then $p$ lies in the smooth locus of $\iss_{\alpha(n,m)} \Z_2^{*n}$ by the Cayley-smoothness of $\C(\Z_2^{*n})$. If $m=n$ and the local quiver setting $(\mathbf{A},\mathbf{k})$ in $p$ has the property $\mathbf{A} = \{N\}$, then it is enough to show that the most degenerate local quiver setting with $\mathbf{A} = \{N\}$ is smooth. The most degenerate local quiver setting with this property is
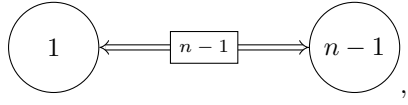
\begin{center}
\begin{tikzpicture}[
    implies/.style={double,double equal sign distance,implies-implies},
    dot/.style={shape=circle,fill=black,minimum size=2pt,
                inner sep=0pt,outer sep=2pt},
                state/.style={circle, draw, minimum size=1.2 cm}]
\node[state] (a) at ( -2, 0) {$1$};
\node[state] (b) at ( 2, 0) {$n-1$};
\tikzset{every node/.style={fill=white},rectangle}
\draw[->,font=\scriptsize]
    (a) edge[implies] node[vertice,rectangle,anchor=center]{$n-1$}(b)
    ;
\end{tikzpicture},
\end{center}
which is one of the smooth quiver settings of theorem \ref{th:symmetric}. So indeed, in this case, all points $p\in \iss_{\alpha(n,n)} \Z_2^{*n}$ with local quiver setting $(\{N\},k)$, $1\leq k \leq n$ are smooth.
\par To show that in all other cases $p$ is a singular point of $\iss_{\alpha(n,m)} \Z_2^{*n}$, we will show that 
\begin{itemize}
\item if $n>2$, if the local quiver setting in $p$ is determined by $(\{A_1,A_2\},(k_1,k_2))$ with $k_1+k_2=m$, then $p$ is a singular point, and
\item if $m<n$ and if the local quiver setting in $p$ is determined by $(\{N\},m-1)$, then $p$ is a singular point.
\end{itemize}
Again, this will be enough, for all local quiver settings are degenerations of these two special cases.
\par Assume that the local quiver setting in $p$ is determined by $(\{A_1,A_2\},(k_1,k_2))$ with $k_1+k_2=m$. Combining  theorem \ref{th:symmetric} and theorem \ref{theorem:main}, if $\iss_{\alpha(n,m)} \Z_2^{*n}$ is smooth in $p$, then this implies that
$$
|A_1|k_2 + |A_2|k_1-k_1k_2 = 1 \Leftrightarrow (|A_1|-\frac{1}{2} k_1)k_2 + (|A_2|-\frac{1}{2} k_2)k_1=1.
$$
The integers $k_1$ and $k_2$ are strictly larger than 0 and both $|A_1|-\frac{1}{2} k_1$ and $|A_2|-\frac{1}{2} k_2$ are strictly larger than 1, so this can never happen. This solves the first case.
\par Assume now that $p$ is a smooth point, $m<n$ and that the local quiver setting in $p$ is determined by  $(\{N\},m-1)$. Again by combining theorem \ref{th:symmetric} and theorem \ref{theorem:main}, this would imply that
$$
n-(m-1) = 1 \Leftrightarrow n=m,
$$
contradicting that $m<n$.
\par In the special case that $n=m=2$, then we have seen in section \ref{sec:tame} that $\rep_{\alpha(2,2)} \Z_2 * \Z_2/\GL_2(\C) = \A^1$, which is clearly smooth.
\end{proof}
\par From the Young degeneration graph (the degeneration graph labelled by Young diagrams and positive integers satisfying the conditions from theorem \ref{theorem:main}) for $\alpha(n,n)$, one can deduce the Young degeneration graph for $\alpha(n,m)$ for any $1\leq m \leq n$, by taking the full subgraph on all vertices with the property that $\sum_{i=1}^l k_l = m$.
\par We can also find the quiver degeneration graph (the degeneration graph labelled by the possible local quiver settings) for $\alpha(n,m)$ for any $1 \leq m \leq n$ from the quiver degeneration graph of $\alpha(n,m)$, by taking all the local quiver settings starting from the $n-m$ column and decreasing the dimension of the vertex corresponding to $\psi_\emptyset$ in the remaining local quiver settings with $n-m$.
\par As an example, in figures \ref{Degenerationgraph3,3young} and \ref{Degenerationgraph4,4young}, the degeneration graphs for $(n,m) = (3,3)$ and $(n,m) = (4,4)$ are depicted using Young diagrams. Equivalently, in figures \ref{Degenerationgraph3,3quiver} and \ref{Degenerationgraph4,4quiver} the corresponding local quiver settings are depicted.
\par In higher dimensions, calculating the entire degeneration graph is quite time-consuming, already for $(n,m) = (5,5)$ there are 24 different local quiver settings. Therefore, it is more suitable to calculate the degeneration graph corresponding to a Young diagram. As an example, in figure \ref{Degenerationgraph333youngtableau}, all possible local quiver settings are shown for the Young diagram $(3,3,3)$.

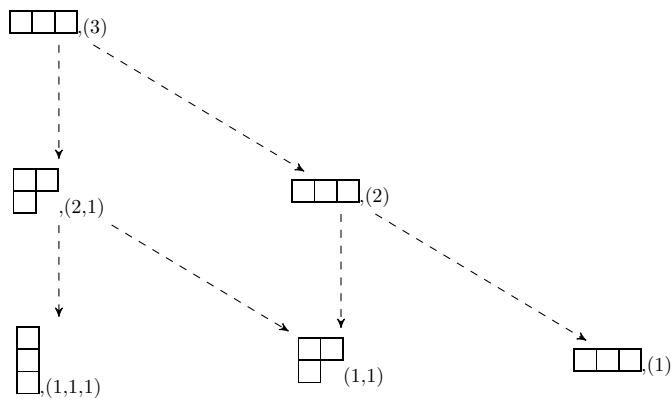
\begin{figure}
\centering 
\begin{tikzpicture}[>=stealth', shorten >=1pt, auto,
    node distance=2.5cm, scale=3/4, 
    transform shape, align=center, 
    state/.style={circle, draw, minimum size=2cm}]
\node (a) at (0,0) {\yng(3),(3)};
\node (b) at (0,-3) {\yng(2,1),(2,1)};
\node (c) at (5,-3) {\yng(3),(2)};
\node (d) at (0,-6) {\yng(1,1,1),(1,1,1)};
\node (e) at (5,-6){\yng(2,1)(1,1)};
\node (f) at (10,-6) {\yng(3),(1)};
\tikzset{every node/.style={fill=white},rectangle}
\draw[->,font=\scriptsize]
    (a) edge[dashed] (b)        
    (a) edge[dashed] (c)        
    (b) edge[dashed] (d)        
    (b) edge[dashed] (e)        
    (c) edge[dashed] (e)        
    (c) edge[dashed] (f)                
    ;
\end{tikzpicture}
\caption{Degeneration graphs of local quiver settings in $\iss_{\alpha(3,3)} \Z_2^{*3}$ with Young diagrams}
\label{Degenerationgraph3,3young}
\end{figure}

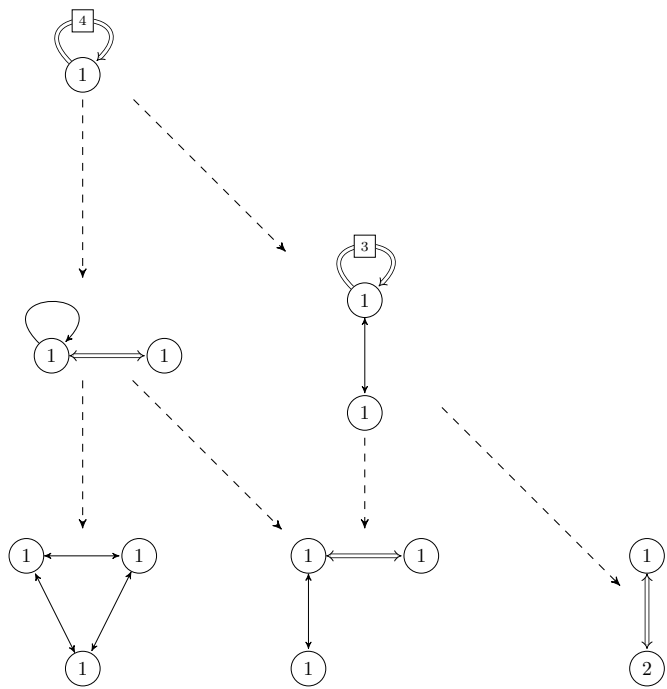
\begin{figure}
\centering
\begin{tikzpicture}[>=stealth', shorten >=1pt, auto,
    node distance=2.5cm, scale=3/4, 
    transform shape, align=center, 
    state/.style={circle, draw, minimum size=2cm}]
\node (3three) at (0,0) {
\begin{tikzpicture}[
    implies/.style={double,double equal sign distance,-implies},
    dot/.style={shape=circle,fill=black,minimum size=2pt,
                inner sep=0pt,outer sep=2pt}]
\node[vertice,circle] (a) at ( 0, 0) {$1$};
\tikzset{every node/.style={fill=white},rectangle}
\draw[->,font=\scriptsize]
    (a) edge[loop above,implies,out=135,in=45,looseness=8] node[vertice,rectangle,anchor=center]{$4$}(a)
    ;
\end{tikzpicture}};
\node (3three21) at (0,-5) {\begin{tikzpicture}[
    implies/.style={double,double equal sign distance,implies-implies},
    dot/.style={shape=circle,fill=black,minimum size=2pt,
                inner sep=0pt,outer sep=2pt}]
\node[vertice,circle] (f) at ( 0, 0) {$1$};
\node[vertice,circle] (g) at ( 2, 0) {$1$};
\tikzset{every node/.style={fill=white},rectangle}
\draw[->,font=\scriptsize]
    (f) edge[loop,out=135,in=45,looseness=8] (f)
    (f) edge[implies] (g)
    ;
\end{tikzpicture}};
\node (2three) at (5,-5) {\begin{tikzpicture}[
    implies/.style={double,double equal sign distance,implies-implies},
    impliesloop/.style={double,double equal sign distance,-implies},
    dot/.style={shape=circle,fill=black,minimum size=2pt,
                inner sep=0pt,outer sep=2pt}]
\node[vertice,circle] (b) at ( 0, 0) {$1$};
\node[vertice,circle] (c) at ( 0, -2) {$1$};
\tikzset{every node/.style={fill=white},rectangle}
\draw[->,font=\scriptsize]
    (b) edge[impliesloop,out=135,in=45,looseness=8] node[vertice,rectangle,anchor=center]{$3$}(b)
    (b) edge[<->] (c)
    ;
\end{tikzpicture}};
\node (3three111) at (0,-10) {\begin{tikzpicture}[anchor=center,
    implies/.style={double,double equal sign distance,implies-implies},
    dot/.style={shape=circle,fill=black,minimum size=2pt,
                inner sep=0pt,outer sep=2pt}]
\node[vertice,circle] (k) at ( -1, 0) {$1$};
\node[vertice,circle] (l) at (  1, 0) {$1$};
\node[vertice,circle] (m) at ( 0, {-2*tan(45)}) {$1$};
\tikzset{every node/.style={fill=white},rectangle}
\draw[->,font=\scriptsize]
    (k) edge[<->] (l)
    (l) edge[<->] (m)
    (m) edge[<->] (k)    
    ;
\end{tikzpicture}};
\node (2three11) at (5,-10){\begin{tikzpicture}[anchor=center,
    implies/.style={double,double equal sign distance,implies-implies},
    dot/.style={shape=circle,fill=black,minimum size=2pt,
                inner sep=0pt,outer sep=2pt}]
\node[vertice,circle] (h) at ( -1, 0) {$1$};
\node[vertice,circle] (j) at (  1, 0) {$1$};
\node[vertice,circle] (i) at ( -1, -2) {$1$};
\tikzset{every node/.style={fill=white},rectangle}
\draw[->,font=\scriptsize]
    (h) edge[<->] (i)
    (h) edge[implies] (j)
    ;
\end{tikzpicture}};
\node (1three) at (10,-10) {\begin{tikzpicture}[
    implies/.style={double,double equal sign distance,implies-implies},
    dot/.style={shape=circle,fill=black,minimum size=2pt,
                inner sep=0pt,outer sep=2pt}]
\node[vertice,circle] (d) at ( 0, 0) {$1$};
\node[vertice,circle] (e) at ( 0, -2) {$2$};
\tikzset{every node/.style={fill=white},rectangle}
\draw[->,font=\scriptsize]
    (d) edge[implies] (e)
    ;
\end{tikzpicture}};
\tikzset{every node/.style={fill=white},rectangle}
\draw[->,font=\scriptsize]  
	(3three) edge[dashed] (3three21)
	(3three21) edge[dashed] (3three111)
	(3three21) edge[dashed] (2three11)            
	(3three) edge[dashed] (2three)     
	(2three) edge[dashed] (2three11)
	(2three) edge[dashed] (1three)       
    ;
\end{tikzpicture}
\caption{Degeneration graphs of local quiver settings in $\iss_{\alpha(3,3)} \Z_2^{*3}$ with quiver settings}
\label{Degenerationgraph3,3quiver}
\end{figure}
\newpage
\begin{sidewaysfigure}
\centering 
\begin{tikzpicture}[>=stealth', shorten >=1pt, auto,
    node distance=2.5cm, scale=3/4, 
    transform shape, align=center, 
    state/.style={circle, draw, minimum size=2cm}]
\node (4four) at (0,0) {\yng(4),(4)};
\node (4four31) at (-2,-5) {\yng(3,1),(3,1)};
\node (4four22) at (2,-5) {\yng(2,2),(2,2)};
\node (4four211) at (0,-10) {\yng(2,1,1),(2,1,1)};
\node (4four1111) at (0,-15) {\yng(1,1,1,1),(1,1,1,1)};
\node (3four) at (8,-5) {\yng(4),(3)};
\node (3four31) at (6,-10) {\yng(3,1),(2,1)};
\node (3four22) at (10,-10) {\yng(2,2),(2,1)};
\node (3four211) at (8,-15) {\yng(2,1,1),(1,1,1)};
\node (2four) at (16,-10) {\yng(4),(2)};
\node (2four31) at (14,-15) {\yng(3,1),(1,1)};
\node (2four22) at (18,-15) {\yng(2,2),(1,1)};
\node (1four) at (24,-15) {\yng(4),(1)};
\tikzset{every node/.style={fill=white},rectangle}
\draw[->,font=\scriptsize]
	(4four) edge[dashed] (4four31)            
	(4four) edge[dashed] (4four22)            
	(4four31) edge[dashed] (4four211)
	(4four22) edge[dashed] (4four211)
	(4four211) edge[dashed] (4four1111)
	(3four) edge[dashed] (3four31)            
	(3four) edge[dashed] (3four22)            
	(3four31) edge[dashed] (3four211)
	(3four22) edge[dashed] (3four211)
	(2four) edge[dashed] (2four31)            
	(2four) edge[dashed] (2four22)
	(4four) edge[dashed] (3four)            
	(3four) edge[dashed] (2four)            
	(2four) edge[dashed] (1four) 
	(4four31) edge[dashed] (3four31)
	(3four31) edge[dashed] (2four31)
	(4four22) edge[dashed] (3four22)
	(3four22) edge[dashed] (2four22)
	(4four211) edge[dashed] (3four211)
    ;
\end{tikzpicture}
\caption{Degeneration graph of local quiver settings in $\iss_{\alpha(4,4)} \Z_2^{*4}$ with Young diagrams}
\label{Degenerationgraph4,4young}
\end{sidewaysfigure}
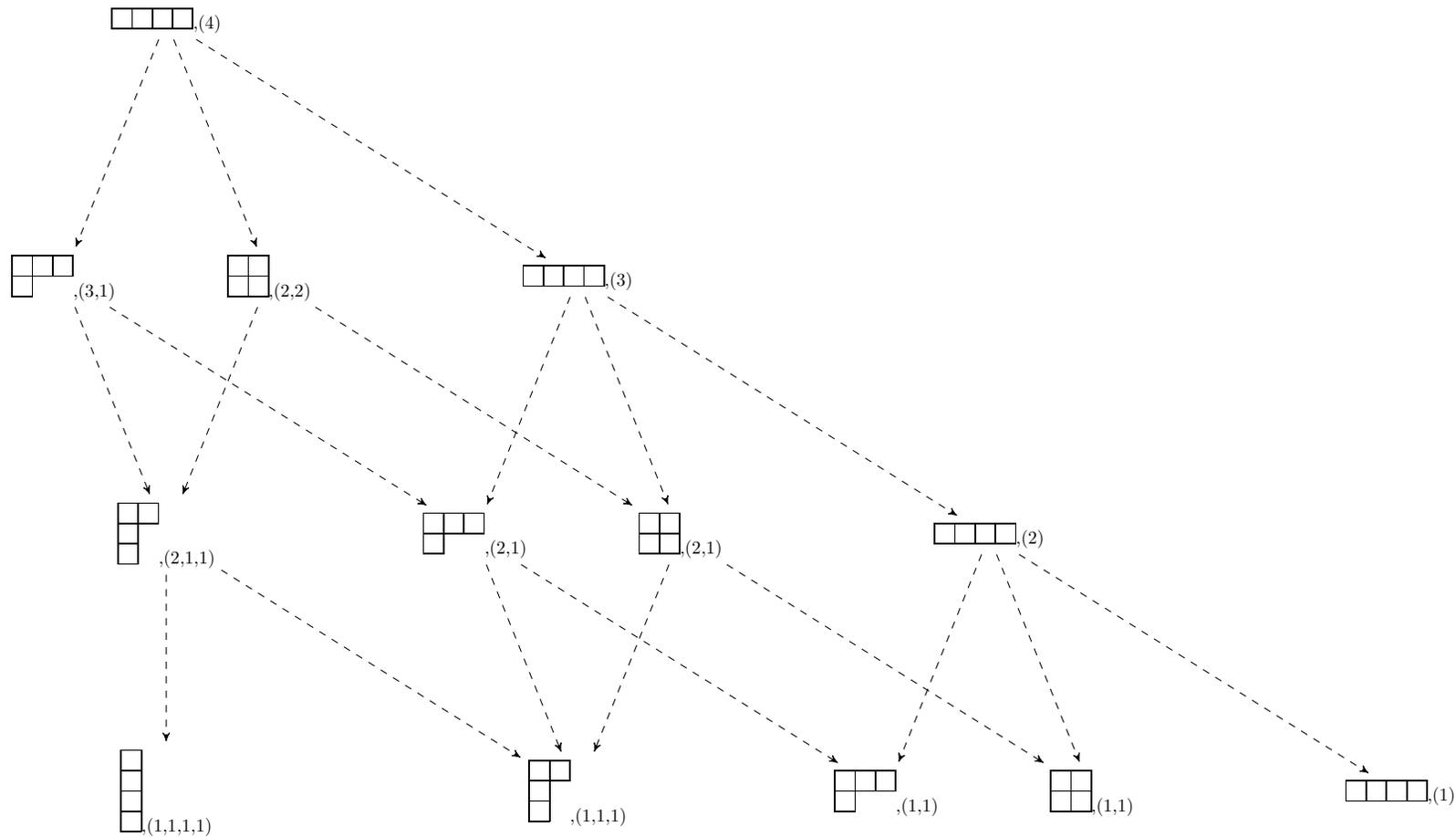
\newpage

\begin{sidewaysfigure}
\centering 
\begin{tikzpicture}[>=stealth', shorten >=1pt, auto,
    node distance=2.5cm, scale=3/4, 
    transform shape, align=center, 
    state/.style={circle, draw, minimum size=2cm}]
\node (4four) at (0,0) {\begin{tikzpicture}[
    implies/.style={double,double equal sign distance,-implies},
    dot/.style={shape=circle,fill=black,minimum size=2pt,
                inner sep=0pt,outer sep=2pt}]
\node[vertice,circle] (a) at ( 0, 0) {$1$};
\tikzset{every node/.style={fill=white},rectangle}
\draw[->,font=\scriptsize]
    (a) edge[loop above,implies,out=135,in=45,looseness=8] node[vertice,rectangle,anchor=center]{$9$}(a)
    ;
\end{tikzpicture}};
\node (4four31) at (-2,-5) {\begin{tikzpicture}[
    implies/.style={double,double equal sign distance,implies-implies},
    impliesloop/.style={double,double equal sign distance,-implies},
    dot/.style={shape=circle,fill=black,minimum size=2pt,
                inner sep=0pt,outer sep=2pt}]
\node[vertice,circle] (c) at ( -1, 0) {$1$};
\node[vertice,circle] (d) at ( 1, 0) {$1$};
\tikzset{every node/.style={fill=white},rectangle}
\draw[->,font=\scriptsize]
    (c) edge[loop above,out=135,in=45,looseness=8,impliesloop]node[vertice,rectangle,anchor=center]{$4$} (c)
    (c) edge[implies]node[vertice,rectangle,anchor=center]{$3$} (d)
    ;
\end{tikzpicture}};
\node (4four22) at (2,-5) {\begin{tikzpicture}[
    implies/.style={double,double equal sign distance,implies-implies},
    dot/.style={shape=circle,fill=black,minimum size=2pt,
                inner sep=0pt,outer sep=2pt}]
\node[vertice,circle] (e) at ( -1, 0) {$1$};
\node[vertice,circle] (f) at (  1, 0) {$1$};
\tikzset{every node/.style={fill=white},rectangle}
\draw[->,font=\scriptsize]
    (e) edge[loop above,out=135,in=45,looseness=8] (e)
    (f) edge[loop above,out=135,in=45,looseness=8] (f)
    (e) edge[implies]node[vertice,rectangle,anchor=center]{$4$} (f)
    ;
\end{tikzpicture}};
\node (4four211) at (0,-10) {\begin{tikzpicture}[anchor=center,
    implies/.style={double,double equal sign distance,implies-implies},
    dot/.style={shape=circle,fill=black,minimum size=2pt,
                inner sep=0pt,outer sep=2pt}]
\node[vertice,circle] (g) at ( -1, 0) {$1$};
\node[vertice,circle] (h) at (  1, 0) {$1$};
\node[vertice,circle] (i) at ( 0, {-2*tan(45)}) {$1$};
\tikzset{every node/.style={fill=white},rectangle}
\draw[->,font=\scriptsize]
    (g) edge[<->] (h)
    (g) edge[implies] (i)
    (h) edge[implies] (i)    
    (i) edge[loop above,out=-135,in=-45,looseness=8] (i)
    ;
\end{tikzpicture}};
\node (4four1111) at (0,-15) {\begin{tikzpicture}[anchor=center,
    implies/.style={double,double equal sign distance,implies-implies},
    dot/.style={shape=circle,fill=black,minimum size=2pt,
                inner sep=0pt,outer sep=2pt}]
\node[vertice,circle] (j) at ( 1, 1) {$1$};
\node[vertice,circle] (k) at ( -1, 1) {$1$};
\node[vertice,circle] (l) at ( -1, -1) {$1$};
\node[vertice,circle] (m) at ( 1, -1) {$1$};
\tikzset{every node/.style={fill=white},rectangle}
\draw[->,font=\scriptsize]
	(j) edge[<->] (k)
	(j) edge[<->] (l)
	(j) edge[<->] (m)
    (k) edge[<->] (l)
    (l) edge[<->] (m)
    (m) edge[<->] (k)    
    ;
\end{tikzpicture}};
\node (3four) at (8,-5) {\begin{tikzpicture}[
    implies/.style={double,double equal sign distance,implies-implies},
    impliesloop/.style={double,double equal sign distance,-implies},
    dot/.style={shape=circle,fill=black,minimum size=2pt,
                inner sep=0pt,outer sep=2pt}]
\node[vertice,circle] (n) at ( 0, 0) {$1$};
\node[vertice,circle] (o) at ( 0, -2) {$1$};
\tikzset{every node/.style={fill=white},rectangle}
\draw[->,font=\scriptsize]
    (n) edge[loop above,impliesloop,out=135,in=45,looseness=8] node[vertice,rectangle,anchor=center]{$8$}(n)
    (n) edge[<->] (o)
    ;
\end{tikzpicture}};
\node (3four31) at (6,-10) {\begin{tikzpicture}[anchor=center,
    implies/.style={double,double equal sign distance,implies-implies},
    impliesloop/.style={double,double equal sign distance,-implies},
    dot/.style={shape=circle,fill=black,minimum size=2pt,
                inner sep=0pt,outer sep=2pt}]
\node[vertice,circle] (r) at ( -1, 0) {$1$};
\node[vertice,circle] (q) at (  1, 0) {$1$};
\node[vertice,circle] (p) at ( -1, -2) {$1$};
\tikzset{every node/.style={fill=white},rectangle}
\draw[->,font=\scriptsize]
    (p) edge[<->] (r)
    (r) edge[implies]node[vertice,rectangle,anchor=center]{$3$} (q)    
    (r) edge[loop above,out=135,in=45,looseness=8,impliesloop]node[vertice,rectangle,anchor=center]{$3$} (r)
    ;
\end{tikzpicture}};
\node (3four22) at (10,-10) {\begin{tikzpicture}[anchor=center,
    implies/.style={double,double equal sign distance,implies-implies},
    dot/.style={shape=circle,fill=black,minimum size=2pt,
                inner sep=0pt,outer sep=2pt}]
\node[vertice,circle] (t) at ( -1, 0) {$1$};
\node[vertice,circle] (u) at (  1, 0) {$1$};
\node[vertice,circle] (s) at ( 1, -2) {$1$};
\tikzset{every node/.style={fill=white},rectangle}
\draw[->,font=\scriptsize]
    (s) edge[<->] (u)
    (t) edge[implies]node[vertice,rectangle,anchor=center]{$4$} (u)    
    (t) edge[loop above,out=135,in=45,looseness=8] (t)
    ;
\end{tikzpicture}};
\node (3four211) at (8,-15) {\begin{tikzpicture}[anchor=center,
    implies/.style={double,double equal sign distance,implies-implies},
    dot/.style={shape=circle,fill=black,minimum size=2pt,
                inner sep=0pt,outer sep=2pt}]
\node[vertice,circle] (v) at ( -1, 0) {$1$};
\node[vertice,circle] (w) at (  1, 0) {$1$};
\node[vertice,circle] (x) at (  0, {-2*tan(45)}) {$1$};
\node[vertice,circle] (y) at ( 0 , {-2-2*tan(45)}) {$1$};
\tikzset{every node/.style={fill=white},rectangle}
\draw[->,font=\scriptsize]
	(x) edge[implies] (v)
	(x) edge[implies] (w)
	(v) edge[<->] (w)
    (x) edge[<->] (y)    
    ;
\end{tikzpicture}};
\node (2four) at (16,-10) {\begin{tikzpicture}[
    implies/.style={double,double equal sign distance,implies-implies},
    impliesloop/.style={double,double equal sign distance,-implies},
    dot/.style={shape=circle,fill=black,minimum size=2pt,
                inner sep=0pt,outer sep=2pt}]
\node[vertice,circle] (z) at ( 0, 0) {$1$};
\node[vertice,circle] (A) at ( 0, -2) {$2$};
\tikzset{every node/.style={fill=white},rectangle}
\draw[->,font=\scriptsize]
    (z) edge[loop above,impliesloop,out=135,in=45,looseness=8] node[vertice,rectangle,anchor=center]{$5$}(z)
    (n) edge[implies] (o)
    ;
\end{tikzpicture}};
\node (2four31) at (14,-15) {\begin{tikzpicture}[anchor=center,
    implies/.style={double,double equal sign distance,implies-implies},
    dot/.style={shape=circle,fill=black,minimum size=2pt,
                inner sep=0pt,outer sep=2pt}]
\node[vertice,circle] (B) at ( -1, 0) {$1$};
\node[vertice,circle] (C) at (  1, 0) {$1$};
\node[vertice,circle] (D) at ( -1, -2) {$2$};
\tikzset{every node/.style={fill=white},rectangle}
\draw[->,font=\scriptsize]
    (B) edge[implies] (D)
    (C) edge[implies]node[vertice,rectangle,anchor=center]{$3$} (B)    
    (D) edge[implies] (B)
    ;
\end{tikzpicture}};
\node (2four22) at (18,-15) {\begin{tikzpicture}[anchor=center,
    implies/.style={double,double equal sign distance,implies-implies},
    dot/.style={shape=circle,fill=black,minimum size=2pt,
                inner sep=0pt,outer sep=2pt}]
\node[vertice,circle] (B) at ( -1, 0) {$1$};
\node[vertice,circle] (C) at (  1, 0) {$1$};
\node[vertice,circle] (D) at ( 0, {-2*tan(45)}) {$2$};
\tikzset{every node/.style={fill=white},rectangle}
\draw[->,font=\scriptsize]
    (C) edge[implies]node[vertice,rectangle,anchor=center]{$3$} (B)    
    (D) edge[<->] (B)
    (D) edge[<->] (C)
    ;
\end{tikzpicture}};
\node (1four) at (24,-15) {\begin{tikzpicture}[
    implies/.style={double,double equal sign distance,implies-implies},
    dot/.style={shape=circle,fill=black,minimum size=2pt,
                inner sep=0pt,outer sep=2pt}]
\node[vertice,circle] (d) at ( 0, 0) {$1$};
\node[vertice,circle] (e) at ( 0, -2) {$3$};
\tikzset{every node/.style={fill=white},rectangle}
\draw[->,font=\scriptsize]
    (d) edge[implies]node[vertice,rectangle,anchor=center]{$3$} (e)
    ;
\end{tikzpicture}};
\tikzset{every node/.style={fill=white},rectangle}
\draw[->,font=\scriptsize]
	(4four) edge[dashed] (4four31)            
	(4four) edge[dashed] (4four22)            
	(4four31) edge[dashed] (4four211)
	(4four22) edge[dashed] (4four211)
	(4four211) edge[dashed] (4four1111)
	(3four) edge[dashed] (3four31)            
	(3four) edge[dashed] (3four22)            
	(3four31) edge[dashed] (3four211)
	(3four22) edge[dashed] (3four211)
	(2four) edge[dashed] (2four31)            
	(2four) edge[dashed] (2four22)
	(4four) edge[dashed] (3four)            
	(3four) edge[dashed] (2four)            
	(2four) edge[dashed] (1four) 
	(4four31) edge[dashed] (3four31)
	(3four31) edge[dashed] (2four31)
	(4four22) edge[dashed] (3four22)
	(3four22) edge[dashed] (2four22)
	(4four211) edge[dashed] (3four211)
    ;
\end{tikzpicture}
\caption{Degeneration graph of local quiver settings in $\iss_{\alpha(4,4)} \Z_2^{*4}$ with quiver settings}
\label{Degenerationgraph4,4quiver}
\end{sidewaysfigure}
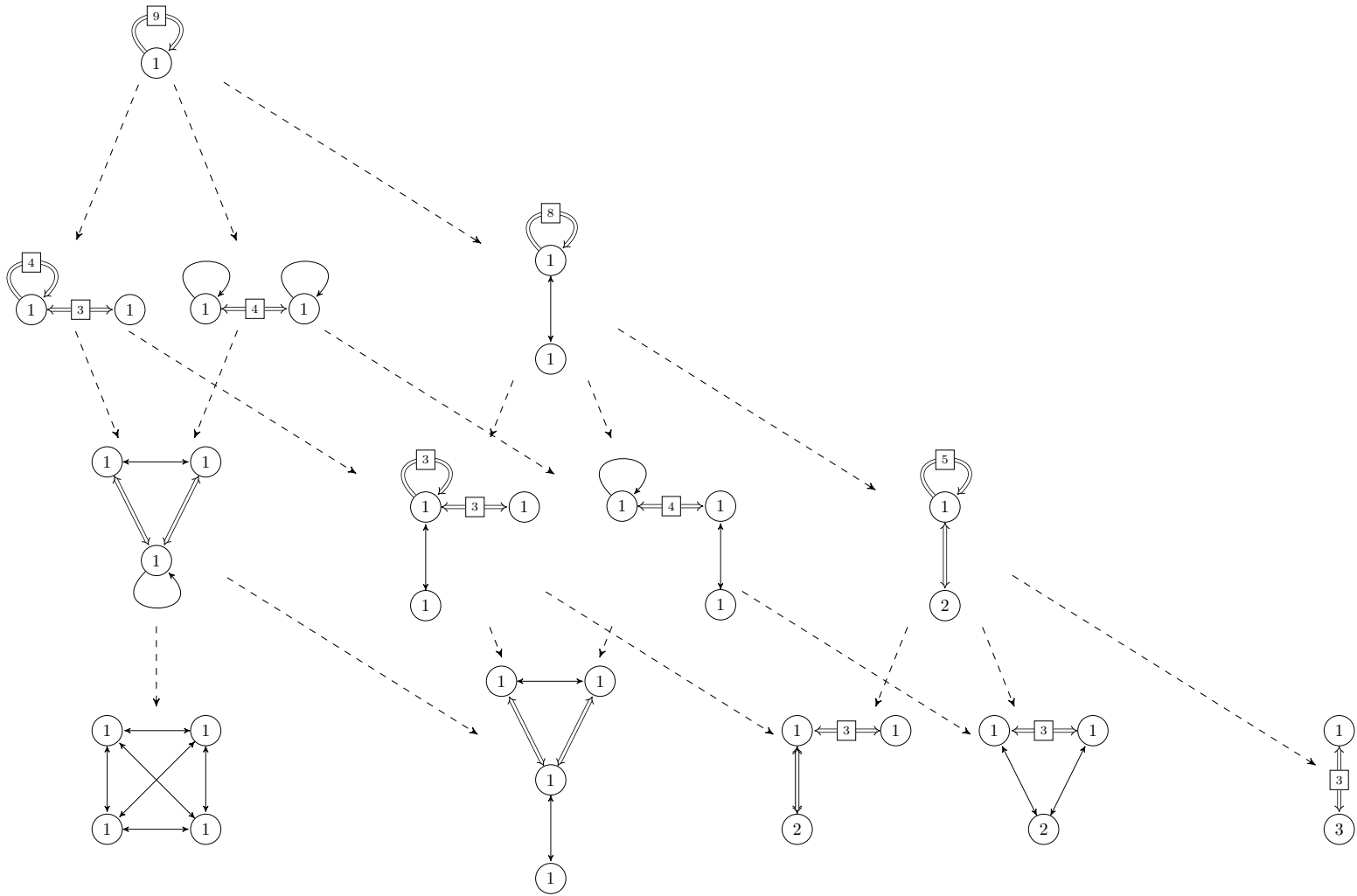
\newpage
\begin{sidewaysfigure}
\centering 
\begin{tikzpicture}[>=stealth', shorten >=1pt, auto,
    node distance=2.5cm, scale=3/4, 
    transform shape, align=center, 
    state/.style={circle, draw, minimum size=2cm}]
\node (A333) at (2,0) {\begin{tikzpicture}[
    implies/.style={double,double equal sign distance,implies-implies},
    impliesloop/.style={double,double equal sign distance,-implies},
    dot/.style={shape=circle,fill=black,minimum size=2pt,
                inner sep=0pt,outer sep=2pt}]
\node[vertice,circle] (a) at ( -1, 0) {$1$};
\node[vertice,circle] (b) at (  1, 0) {$1$};
\node[vertice,circle] (c) at ( 0, {-2*tan(45)}) {$1$};
\tikzset{every node/.style={fill=white},rectangle}
\draw[->,font=\scriptsize]
    (a) edge[implies]node[vertice,rectangle,anchor=center]{$9$} (b)
    (a) edge[implies]node[vertice,rectangle,anchor=center]{$9$} (c)
    (b) edge[implies]node[vertice,rectangle,anchor=center]{$9$} (c)    
    (a) edge[loop,impliesloop,out =-135,in=135,looseness=6]node[vertice,rectangle,anchor=center]{$4$} (a)
    (b) edge[loop,impliesloop,out =-45,in=45,looseness=6]node[vertice,rectangle,anchor=center]{$4$} (b)
	(c) edge[loop,impliesloop,out =-135,in=-45,looseness=6]node[vertice,rectangle,anchor=center]{$4$} (c)    
        ;
\end{tikzpicture}};
\node (A332) at (7,0) {\begin{tikzpicture}[
    implies/.style={double,double equal sign distance,implies-implies},
    impliesloop/.style={double,double equal sign distance,-implies},
    dot/.style={shape=circle,fill=black,minimum size=2pt,
                inner sep=0pt,outer sep=2pt}]
\node[vertice,circle] (d) at ( -1, 0) {$1$};
\node[vertice,circle] (e) at (  1, 0) {$1$};
\node[vertice,circle] (f) at ( 0, {-2*tan(45)}) {$1$};
\node[vertice,circle] (g) at ( 0, {-2/3*tan(45)}) {$1$};
\tikzset{every node/.style={fill=white},rectangle}
\draw[->,font=\scriptsize]
    (d) edge[implies]node[vertice,rectangle,anchor=center]{$9$} (e)
    (e) edge[implies]node[vertice,rectangle,anchor=center]{$9$} (f)
    (f) edge[implies]node[vertice,rectangle,anchor=center]{$9$} (d)    
    (d) edge[loop,impliesloop,out =-135,in=135,looseness=6]node[vertice,rectangle,anchor=center]{$4$} (d)
    (e) edge[loop,impliesloop,out =-45,in=45,looseness=6]node[vertice,rectangle,anchor=center]{$4$} (e)
	(f) edge[loop,impliesloop,out =-135,in=-45,looseness=6]node[vertice,rectangle,anchor=center]{$3$} (f)    
	(f) edge[<->] (g)
        ;
\end{tikzpicture}};
\node (A331) at (10,-5) {\begin{tikzpicture}[
    implies/.style={double,double equal sign distance,implies-implies},
    impliesloop/.style={double,double equal sign distance,-implies},
    dot/.style={shape=circle,fill=black,minimum size=2pt,
                inner sep=0pt,outer sep=2pt}]
\node[vertice,circle] (h) at ( -1, 0) {$1$};
\node[vertice,circle] (i) at (  1, 0) {$1$};
\node[vertice,circle] (j) at ( 0, {-2*tan(45)}) {$1$};
\node[vertice,circle] (k) at ( 0, {-2/3*tan(45)}) {$2$};
\tikzset{every node/.style={fill=white},rectangle}
\draw[->,font=\scriptsize]
    (h) edge[implies]node[vertice,rectangle,anchor=center]{$9$} (i)
    (i) edge[implies]node[vertice,rectangle,anchor=center]{$9$} (j)
    (j) edge[implies]node[vertice,rectangle,anchor=center]{$9$} (h)    
    (h) edge[loop,impliesloop,out =-135,in=135,looseness=6]node[vertice,rectangle,anchor=center]{$4$} (h)
    (i) edge[loop,impliesloop,out =-45,in=45,looseness=6]node[vertice,rectangle,anchor=center]{$4$} (i)
	(j) edge[implies] (k)
        ;
\end{tikzpicture}};
\node (A322) at (10,5) {\begin{tikzpicture}[
    implies/.style={double,double equal sign distance,implies-implies},
    impliesloop/.style={double,double equal sign distance,-implies},
    dot/.style={shape=circle,fill=black,minimum size=2pt,
                inner sep=0pt,outer sep=2pt}]
\node[vertice,circle] (l) at ( -1, 0) {$1$};
\node[vertice,circle] (m) at (  1, 0) {$1$};
\node[vertice,circle] (n) at ( 0, {-2*tan(45)}) {$1$};
\node[vertice,circle] (o) at ( 0, {-2/3*tan(45)}) {$2$};
\tikzset{every node/.style={fill=white},rectangle}
\draw[->,font=\scriptsize]
    (l) edge[implies]node[vertice,rectangle,anchor=center]{$9$} (m)
    (m) edge[implies]node[vertice,rectangle,anchor=center]{$9$} (n)
    (n) edge[implies]node[vertice,rectangle,anchor=center]{$8$} (l)    
    (l) edge[loop,impliesloop,out =-135,in=135,looseness=6]node[vertice,rectangle,anchor=center]{$3$} (l)
    (m) edge[loop,impliesloop,out =-45,in=45,looseness=6]node[vertice,rectangle,anchor=center]{$4$} (m)
	(n) edge[loop,impliesloop,out =-135,in=-45,looseness=6]node[vertice,rectangle,anchor=center]{$3$} (n)    
	(l) edge[<->] (o)
	(n) edge[<->] (o)
        ;
\end{tikzpicture}};
\node (A321) at (15,-5) {\begin{tikzpicture}[
    implies/.style={double,double equal sign distance,implies-implies},
    impliesloop/.style={double,double equal sign distance,-implies},
    dot/.style={shape=circle,fill=black,minimum size=2pt,
                inner sep=0pt,outer sep=2pt}]
\node[vertice,circle] (p) at ( -1, 0) {$1$};
\node[vertice,circle] (q) at (  1, 0) {$1$};
\node[vertice,circle] (r) at ( 0, {-2*tan(45)}) {$1$};
\node[vertice,circle] (s) at ( 0, {-2/3*tan(45)}) {$3$};
\tikzset{every node/.style={fill=white},rectangle}
\draw[->,font=\scriptsize]
    (p) edge[implies]node[vertice,rectangle,anchor=center]{$9$} (q)
    (q) edge[implies]node[vertice,rectangle,anchor=center]{$9$} (r)
    (r) edge[implies]node[vertice,rectangle,anchor=center]{$6$} (p)    
    (p) edge[loop,impliesloop,out =-135,in=135,looseness=6]node[vertice,rectangle,anchor=center]{$3$} (p)
    (q) edge[loop,impliesloop,out =-45,in=45,looseness=6]node[vertice,rectangle,anchor=center]{$4$} (q)
	(p) edge[<->] (s)
	(r) edge[implies] (s)
        ;
\end{tikzpicture}};
\node (A222) at (15,5) {\begin{tikzpicture}[
    implies/.style={double,double equal sign distance,implies-implies},
    impliesloop/.style={double,double equal sign distance,-implies},
    dot/.style={shape=circle,fill=black,minimum size=2pt,
                inner sep=0pt,outer sep=2pt}]
\node[vertice,circle] (t) at ( -1, 0) {$1$};
\node[vertice,circle] (u) at (  1, 0) {$1$};
\node[vertice,circle] (v) at ( 0, {-2*tan(45)}) {$1$};
\node[vertice,circle] (w) at ( 0, {-2/3*tan(45)}) {$3$};
\tikzset{every node/.style={fill=white},rectangle}
\draw[->,font=\scriptsize]
    (t) edge[implies]node[vertice,rectangle,anchor=center]{$8$} (u)
    (v) edge[implies]node[vertice,rectangle,anchor=center]{$8$} (u)
    (v) edge[implies]node[vertice,rectangle,anchor=center]{$8$} (t)    
    (t) edge[loop,impliesloop,out =-135,in=135,looseness=6]node[vertice,rectangle,anchor=center]{$3$} (t)
    (u) edge[loop,impliesloop,out =-45,in=45,looseness=6]node[vertice,rectangle,anchor=center]{$3$} (u)
	(v) edge[loop,impliesloop,out =-135,in=-45,looseness=6]node[vertice,rectangle,anchor=center]{$3$} (v)    
	(t) edge[<->] (w)
	(u) edge[<->] (w)
	(v) edge[<->] (w)
        ;
\end{tikzpicture}};
\node (A311) at (20,-5) {\begin{tikzpicture}[
    implies/.style={double,double equal sign distance,implies-implies},
    impliesloop/.style={double,double equal sign distance,-implies},
    dot/.style={shape=circle,fill=black,minimum size=2pt,
                inner sep=0pt,outer sep=2pt}]
\node[vertice,circle] (B) at ( -1, 0) {$1$};
\node[vertice,circle] (C) at (  1, 0) {$1$};
\node[vertice,circle] (D) at ( 0, {-2*tan(45)}) {$1$};
\node[vertice,circle] (E) at ( 0, {-2/3*tan(45)}) {$4$};
\tikzset{every node/.style={fill=white},rectangle}
\draw[->,font=\scriptsize]
    (B) edge[implies]node[vertice,rectangle,anchor=center]{$9$} (C)
    (C) edge[implies]node[vertice,rectangle,anchor=center]{$9$} (D)
    (D) edge[implies]node[vertice,rectangle,anchor=center]{$5$} (B)    
    (C) edge[loop,impliesloop,out =-45,in=45,looseness=6]node[vertice,rectangle,anchor=center]{$4$} (C)
	(B) edge[implies] (E)
	(D) edge[implies] (E)
        ;
\end{tikzpicture}};
\node (A221) at (20,5) {\begin{tikzpicture}[
    implies/.style={double,double equal sign distance,implies-implies},
    impliesloop/.style={double,double equal sign distance,-implies},
    dot/.style={shape=circle,fill=black,minimum size=2pt,
                inner sep=0pt,outer sep=2pt}]
\node[vertice,circle] (x) at ( -1, 0) {$1$};
\node[vertice,circle] (y) at (  1, 0) {$1$};
\node[vertice,circle] (z) at ( 0, {-2*tan(45)}) {$1$};
\node[vertice,circle] (A) at ( 0, {-2/3*tan(45)}) {$4$};
\tikzset{every node/.style={fill=white},rectangle}
\draw[->,font=\scriptsize]
    (x) edge[implies]node[vertice,rectangle,anchor=center]{$8$} (y)
    (x) edge[implies]node[vertice,rectangle,anchor=center]{$7$} (z)
    (z) edge[implies]node[vertice,rectangle,anchor=center]{$7$} (y)    
    (x) edge[loop,impliesloop,out =-135,in=135,looseness=6]node[vertice,rectangle,anchor=center]{$3$} (x)
    (y) edge[loop,impliesloop,out =-45,in=45,looseness=6]node[vertice,rectangle,anchor=center]{$3$} (y)
	(x) edge[<->] (A)
	(y) edge[<->] (A)
	(z) edge[implies] (A)
        ;
\end{tikzpicture}};
\node (A211) at (23,0) {\begin{tikzpicture}[
    implies/.style={double,double equal sign distance,implies-implies},
    impliesloop/.style={double,double equal sign distance,-implies},
    dot/.style={shape=circle,fill=black,minimum size=2pt,
                inner sep=0pt,outer sep=2pt}]
\node[vertice,circle] (F) at ( -1, 0) {$1$};
\node[vertice,circle] (G) at (  1, 0) {$1$};
\node[vertice,circle] (H) at ( 0, {-2*tan(45)}) {$1$};
\node[vertice,circle] (I) at ( 0, {-2/3*tan(45)}) {$5$};
\tikzset{every node/.style={fill=white},rectangle}
\draw[->,font=\scriptsize]
    (F) edge[implies]node[vertice,rectangle,anchor=center]{$7$} (G)
    (G) edge[implies]node[vertice,rectangle,anchor=center]{$7$} (H)	
    (H) edge[implies]node[vertice,rectangle,anchor=center]{$5$} (F)    
    (G) edge[loop,impliesloop,out =-45,in=45,looseness=6]node[vertice,rectangle,anchor=center]{$3$} (G)
	(F) edge[implies] (I)
	(G) edge[<->] (I)
	(H) edge[implies] (I)
        ;
\end{tikzpicture}};
\node (A111) at (28,0) {\begin{tikzpicture}[
    implies/.style={double,double equal sign distance,implies-implies},
    impliesloop/.style={double,double equal sign distance,-implies},
    dot/.style={shape=circle,fill=black,minimum size=2pt,
                inner sep=0pt,outer sep=2pt}]
\node[vertice,circle] (J) at ( -1, 0) {$1$};
\node[vertice,circle] (K) at (  1, 0) {$1$};
\node[vertice,circle] (L) at ( 0, {-2*tan(45)}) {$1$};
\node[vertice,circle] (M) at ( 0, {-2/3*tan(45)}) {$6$};
\tikzset{every node/.style={fill=white},rectangle}
\draw[->,font=\scriptsize]
    (J) edge[implies]node[vertice,rectangle,anchor=center]{$5$} (K)
    (K) edge[implies]node[vertice,rectangle,anchor=center]{$5$} (L)	
    (L) edge[implies]node[vertice,rectangle,anchor=center]{$5$} (J)    
	(J) edge[implies] (M)
	(K) edge[implies] (M)
	(L) edge[implies] (M)
        ;
\end{tikzpicture}};
\tikzset{every node/.style={fill=white},rectangle}
\draw[->,font=\scriptsize]
    (A333) edge[dashed] (A332)
    (A332) edge[dashed] (A331)
    (A332) edge[dashed] (A322)    
    (A331) edge[dashed] (A321)
    (A322) edge[dashed] (A321)
    (A322) edge[dashed] (A222)
    (A321) edge[dashed] (A221)
    (A321) edge[dashed] (A311)
    (A222) edge[dashed] (A221)
    (A221) edge[dashed] (A211)
    (A311) edge[dashed] (A211)
    (A211) edge[dashed] (A111)
    ;
\end{tikzpicture}
\caption{Degeneration graph of local quiver settings in $\iss_{\alpha(9,9)} \Z_2^{*9}$,(3,3,3)-diagonal}
\label{Degenerationgraph333youngtableau}
\end{sidewaysfigure}
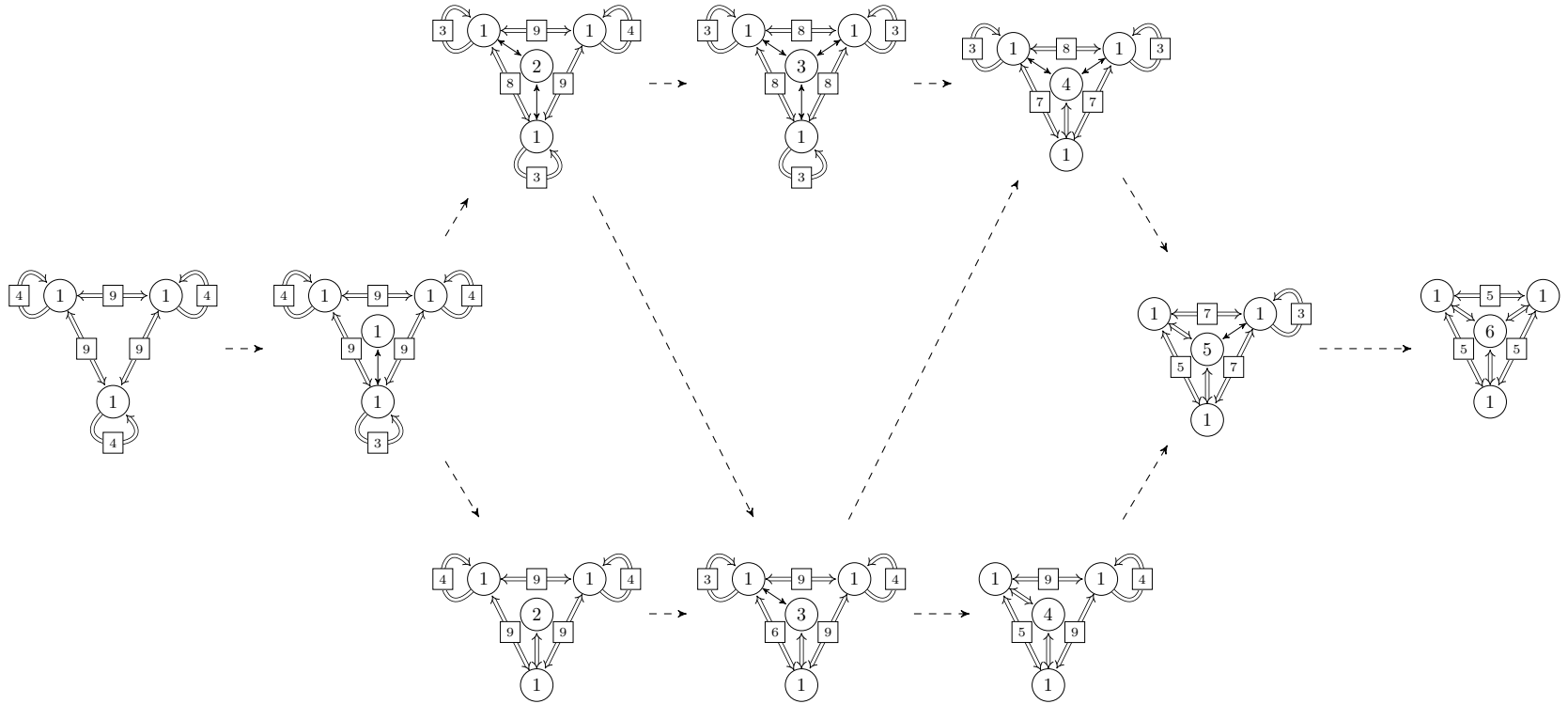

\end{document}